\documentclass{amsart}

\usepackage{color}

\newtheorem{Theorem}{Theorem}

\newtheorem{Lemma}{Lemma}
\newtheorem{Corollary}{Corollary}

\newcommand{\im}{\operatorname{im}}

\newcommand{\Eam}{\Omega}
\newcommand{\inte}{{\mathrm{int}}\,}
\newcommand{\rd}{{\mathbb R}^d}
\newcommand{\R}{{\mathbb R}}
\newcommand{\N}{{\mathbb N}}
\newcommand{\BK}{{\mathbb K}}

\newcommand{\cal}{\mathcal}
\newcommand{\ep}{\varepsilon}
\newcommand{\Lin}{{\rm Lin}}
\newcommand{\Tan}{{\rm Tan}}
\newcommand{\nor}{{\rm nor}}

\newcommand{\Forth}{F^{\perp}}
\newcommand{\cI}{{\mathcal I}}
\newcommand{\bark}{d-1-k}

\newcommand{\ve}{\varepsilon}
\newcommand{\bK}{{\mathbb K}}
\newcommand{\Ha}{{\cal H}}
\newcommand{\ap}{\operatorname{ap}}
\newcommand{\Nor}{{\mathrm{Nor}}\,}
\newcommand{\ul}{\underline}

\DeclareMathOperator*{\bigtimes}{\mbox{\LARGE{$\times$}}}
\allowdisplaybreaks
\title{Flag representations of mixed volumes and mixed functionals of convex bodies}

\author{Daniel Hug}
\thanks{The first and third authors are grateful for support from the German Science Foundation (DFG). The second author was supported by the Czech Science Foundation, project No. P201/15-08218S}
\address{Karls\-ruhe Institute of Technology, Department of Mathematics,
D-76128 Karls\-ruhe, Germany}
\email{daniel.hug@kit.edu}
\urladdr{http://www.math.kit.edu/$\sim$hug/}

\author{Jan Rataj}
\address{Charles University, Faculty of Mathematics and Physics,
Sokolovska 83, 186 75 Praha 8,
Czech Republic}
\email{rataj@karlin.mff.cuni.cz}
\urladdr{http://www.karlin.mff.cuni.cz/$\sim$rataj/index\underline{ }en.html}

\author{Wolfgang Weil}
\address{Karls\-ruhe Institute of Technology, Department of Mathematics,
D-76128 Karls\-ruhe, Germany}
\email{wolfgang.weil@kit.edu}
\urladdr{http://www.math.kit.edu/$\sim$weil/}

\date{September 19, 2017}
\begin{document}

\begin{abstract}
Mixed volumes $V(K_1,\dots, K_d)$ of convex bodies $K_1,\dots ,K_d$ in Euclidean space $\mathbb{R}^d$ are of central importance in the Brunn-Minkowski theory.
Representations for  mixed volumes are available in special cases, for example
as integrals over the unit sphere with respect to mixed area measures. More generally, in Hug-Rataj-Weil (2013) a formula for $V(K [n], M[d-n])$, $n\in \{1,\dots ,d-1\}$, as a double integral over flag manifolds was established which involved certain flag measures of the convex bodies $K$ and $M$ (and required a general position of the bodies). In the following, we discuss the general case $V(K_1[n_1],\dots , K_k[n_k])$, $n_1+\cdots +n_k=d$, and show a corresponding result involving the flag measures $\Omega_{n_1}(K_1;\cdot),\dots, \Omega_{n_k}(K_k;\cdot)$. For this purpose, we first establish a curvature representation of mixed volumes over the normal bundles of the bodies involved.

We also obtain a corresponding flag representation for the mixed functionals from translative integral geometry and a local version, for mixed (translative) curvature measures.
\end{abstract}

\maketitle
\markright{FLAG REPRESENTATIONS OF MIXED FUNCTIONALS}







\section{Introduction}\label{1}

Mixed volumes of convex bodies  build a basic concept and tool
in the Brunn-Minkowski theory of convex geometry. They arise by combining two fundamental geometric notions, the Minkowski addition of sets and the volume functional $V_d$. Namely, for
convex bodies $K_1,\dots ,K_k$ (non-empty compact convex sets) in $\R^d, d\ge 2$, and numbers $t_1,\dots ,t_k\ge 0,$ the volume of the linear combination $t_1K_1+\cdots +t_kK_k$ (which is again a convex body) is a (homogeneous) polynomial in $t_1,\dots ,t_k$, that is
\begin{equation}\label{polexp}
V_d(t_1K_1+\cdots +t_kK_k) = \sum_{i_1=1}^k\cdots\sum_{i_d=1}^k t_{i_1}\cdots t_{i_d}V(K_{i_1},\dots,K_{i_d}).
\end{equation}
The coefficients $V(K_{i_1},\dots,K_{i_d})$ are assumed to be symmetric and are therefore uniquely determined. Moreover, $V(K_{i_1},\dots,K_{i_d})$ is linear in each of its entries $K_{i_1},\dots,K_{i_d}$. For further basic properties of mixed volumes and all other notions from convex geometry which we use, we refer to the book \cite{S14}. As usual, we abbreviate by $V(K_1[n_1],\dots , K_k[n_k])$ the mixed volume where the body $K_i$ appears $n_i$ times, for $i=1,\ldots,k$, and $n_1+\cdots +n_k=d$. The functional $V(K_1[n_1],\dots , K_k[n_k])$ is homogeneous of degree $n_i$ in $K_i$.

In \cite{HRW}, it was shown that
\begin{eqnarray} \label{mv4}
&&V(K[n],M[d-n])\nonumber\\
&&\qquad\qquad \,=\iint f_{n,d-n}(u,U,v,V)\, \Eam_n(K;d(u,U))\,\Eam_{d-n}(M;d(v,V)) ,
\end{eqnarray}
where $\Eam_n(K;\cdot)$ and $\Eam_{d-n}(M;\cdot)$ are flag measures of $K$ and $M$, respectively, the function $f_{n,d-n}$ is independent of $K$ and $M$, and the integration is over the manifold of flags $(u,U)$  (respectively $(v,V)$).
 For this formula,  we had to assume
 that $K$ and $M$ are in general relative position with respect to each other. If $K$ and $M$ are
 polytopes, this condition is, for instance, satisfied if $K$ and $M$ do not have parallel faces of complementary
 dimension.  The proof of \eqref{mv4} was based on a curvature representation of mixed functionals from translative integral geometry which was proved in \cite{RZ95} and used the fact that the mixed volume $V(K[n],-M[d-n])$
 and the mixed functional $V_{n,d-n}(K,M)$ from translative integral geometry coincide (up to a binomial coefficient).

 The iteration of translative integral formulas yields an expansion which resembles \eqref{polexp} but involves mixed functionals of a different nature. Namely,
\begin{align}\label{transk}
\int_{\rd}\cdots\int_{\rd}&V_j(K_1\cap (K_2+z_2)\cap\dots\cap(K_k+z_k))\, {\cal H}^d(dz_k)\cdots{\cal H}^d(dz_2)\nonumber\\ &=\sum_{\substack{r_1,\dots, r_k=j\\r_1+\dots +r_k=(k-1)d+j}}^{d}V_{r_1,\dots,r_k}(K_1,\dots ,K_k)
\end{align}
for $j=0,\dots ,d$, where ${\cal H}^d$ denotes the $d$-dimensional Hausdorff measure. Translative integral formulas are at the basis of integral geometry and have important applications in stochastic geometry. We refer to
\cite[Section 6.4 and Chapter 9]{SW}, for background information, and for details of such applications and for further references. Since $j$ is determined by $j=r_1+\dots +r_k-(k-1)d$,  we skipped the upper index $(j)$ which was used in \cite{SW} and previous papers for the mixed functionals on the right-hand side of \eqref{transk}. We remark that $V_{r_1,\dots,r_k}(K_1,\dots ,K_k)$ is symmetric in the bodies involved, as long as $K_1,\dots, K_k$ and $r_1,\dots ,r_k$ undergo the same permutation. Moreover, if $r_i=0$ (hence $j=0$), then the mixed functional $V_{r_1,\dots,r_k}(K_1,\dots ,K_k)$ does not depend on $K_i$, and if $r_k=d$, then
$$
V_{r_1,\dots,r_{k}}(K_1,\dots ,K_k) = V_{r_1,\dots,r_{k-1}}(K_1,\dots ,K_{k-1})V_d(K_k).
$$
Hence, we may concentrate on the cases where $1\le r_1,\dots ,r_k\le d-1$.
Since $V_{r_1,\dots,r_k}(K_1,\dots ,K_k)$ is homogeneous of degree $r_i$ in $K_i$, $i=1,\dots k$, the total degree of the mixed functional is $r_1+\dots +r_k= (k-1)d+j$. Therefore, for $k>2$ or for $k=2$ and $j>0$, the mixed volume $V(K_1[n_1],\dots , K_k[n_k])$ and the mixed functional $V_{r_1,\dots ,r_k}(K_1,\dots, K_k)$ have completely different homogeneity properties. In fact, apart from the case $k=2$ mentioned above, no simple connection between mixed volumes and mixed translative functionals is known. As a consequence, the curvature representation for mixed translative functionals, which was established in \cite{Hug} (see also \cite{HR}) cannot be used directly for the mixed volume $V(K_1[n_1],\dots , K_k[n_k])$. It is our first goal to provide such a result for mixed volumes.

After collecting some basic facts from convex geometry in Section \ref{2}, we will derive this curvature representation in Section \ref{3} (based on results from \cite{Hug}).  In Section \ref{4}, we discuss the special case of polytopes and relate the curvature representation of mixed volumes  to a formula of Schneider \cite{S94}. Our main result, the flag representation of mixed volumes is formulated and proved in Section \ref{5}. The next Section \ref{6} contains a corresponding flag representation of the mixed translative functionals. In the final Section \ref{7} we discuss a local version of the latter result.

The flag representations of mixed volumes and mixed functionals are useful for applications in stochastic geometry. In particular, for stationary non-isotropic Boolean models $Y$ in ${\mathbb R}^d$, it was recently shown in \cite{HW} that the specific mixed volumes of $Y$ (mean values with respect to convex test bodies $K$) determine the intensity of the underlying particle system uniquely. The proof makes use of our integral representations and shows that even the specific flag measures of the particles are determined.

\section{Basic facts}\label{2}
Let $\rd$ be the $d$-dimensional Euclidean space with scalar product $\langle \cdot\, ,\cdot\rangle$ and norm $\|\cdot\|$. The unit ball
and the unit sphere of $\rd$ are denoted by $B^d$ and $S^{d-1}$, respectively. We put
$$
S_+^{d-1}:=\{ u=(u^1,\dots, u^d)\in S^{d-1} : u^i>0\} .
$$
For $x\in\rd$ and a linear subspace $U\subset\rd$, let $U^\perp$ denote the
orthogonal complement of $U$, $p_Ux=x|U$ the orthogonal projection of $x$ onto $U$, and $p_UA=A|U$
the orthogonal projection of a set $A\subset\rd$ onto $U$. Moreover, we write $\partial(A|U)$ for the
topological boundary of $A|U$ with respect to $U$ as the ambient space.

For a given $k\in\{0,\ldots,d\}$, we denote by  $\bigwedge_k\rd$
the $\binom{d}{k}$-dimensional linear space of $k$-vectors in $\rd$. As usual, we identify $\bigwedge_0\rd$ with $\R$.
The vector space  $\bigwedge_k\rd$ is equipped with the scalar product $\langle\cdot\, ,\cdot\rangle$ as described
in \cite[\S 1.7.5]{Federer69}. We refer to \cite{HRW} for further details and notions which we shall use in the following
and in particular to \cite[Chapter 1]{Federer69} for a brief introduction to multilinear algebra.

The $j$-dimensional Hausdorff measure in a metric space
will be denoted by $\mathcal{H}^j$ with the same normalization as in \cite[\S 2.10.2, p.~171]{Federer69}.
Let $\nu_k^d$ denote the $O(d)$ invariant measure on the Grassmannian of $k$-dimensional linear subspaces $G(d,k)$
of $\R^d$, normalized to a probability measure. We put $\kappa_k := \mathcal{H}^k(B^k)$, $k\in\mathbb{N}_0$, and $\omega_k := k\kappa_k=\mathcal{H}^{k-1}(S^{k-1})$ for $k\in \N$.

In the following, we repeatedly make  use of notions and basic results of geometric measure theory such as the coarea formula which requires the notions of an approximate differential and of an approximate Jacobian. A general form of the coarea formula is for instance stated in \cite[Theorem 3.2.22]{Federer69}, approximate differentials are introduced in \cite[page 253]{Federer69} and the approximate Jacobian is defined in \cite[Theorem 3.2.22]{Federer69} (see also \cite{EG} and \cite{KP}).

Let ${\cal K}$ be the class of all convex bodies in ${\mathbb R}^d$. For  $K \in {\cal K}$ with boundary $\partial K$,  let
$$\nor (K):=\{ (x,u)\in\partial K\times S^{d-1}:\,\langle u  , y-x\rangle\leq 0\mbox{ for all } y\in K\}$$
be its unit normal bundle. This is a $(d-1)$-rectifiable set. Moreover, $\Nor(K,x)$ denotes the normal cone of $K$ at $x\in K$ (we have $\Nor(K,x)=\{0\}$, if $x\in \text{int}(K):= K\setminus \partial K$).

The $k$th {\it support measure} $\Xi_k(K;\cdot)$ of $K$ is a measure on $\R^d\times S^{d-1}$
which is concentrated on $\nor( K)$ and defined by
\begin{align}\label{suppmeas}
&\int g(x,u)\,\Xi_k(K;d(x,u))\nonumber\\
&\qquad =\frac 1{\omega_{d-k}}\int_{\nor (K)}g(x,u)\sum_{|I|=
d-1-k}
\BK_I(K;x,u)\,{\cal H}^{d-1}(d(x,u)),
\end{align}
where $g$ is any bounded measurable function on $\rd\times\ S^{d-1}$, $I$ denotes a subset of $\{ 1,\ldots ,d-1\}$ of cardinality $|I|$,
$$\BK_I(K;x,u):=\frac{\prod_{i\in I}k_i(K;x,u)}{\prod_{i=1}^{d-1}\sqrt{1+k_i(K;x,u)^2}},$$
and the numbers $k_i(K;x,u)\in [0,\infty]$ are the generalized principal curvatures of $K$ at $(x,u)\in\nor (K)$, $i=1,\ldots ,d-1$. If $k_i(K;x,u)=\infty$ for some $i\in\{1,\ldots,d-1\}$, then $\BK_I(K;x,u)$ is determined as the limit which is obtained as
$k_i(K;x,u)\to \infty$. In particular, this implies $\frac 1{\sqrt{1+\infty^2}}=0$ and $\frac \infty{\sqrt{1+\infty^2}}=1$. Moreover, a  product over an empty index set is considered as a factor one. The generalized principal curvatures are defined for ${\cal H}^{d-1}$-almost all $(x,u)\in \nor (K)$. We refer to \cite{Zaehle86}, \cite{Hug98} and \cite{S14} for background information and an introduction to these generalized curvatures and measures from the viewpoint of geometric measure theory.
We also use the notation
$$
A_I(K;x,u):=\Lin\{ a_i(K;x,u):\, i\in I\},
$$
where $a_i(K;x,u)\in S^{d-1}$, $i=1,\ldots ,d-1$, is a generalized principal direction of curvature of $K$ at $(x,u)$, corresponding to the generalized principal curvature $k_i(K;x,u)$, and the vectors $a_1(K;x,u), \dots, a_{d-1}(K;x,u)$ form an orthonormal basis of $u^\perp$ (the subspace orthogonal to $u$). Here, $\Lin$ denotes the linear hull. If $I=\emptyset$, then $A_I(K;x,u)=\{0\}$.   Sometimes it is convenient to consider $A_I(K;x,u)$ as a multivector  (cf.\ Section \ref{3}), i.e.\
$$
A_I(K;x,u)=\textstyle{\bigwedge}_{i\in I}a_i(K;x,u).
$$
Here, the right-hand side is $1\in\bigwedge_0\R^d$ if $I=\emptyset$.

The support measures $\Xi_k(K;\cdot)$ also arise as coefficients in a local Steiner formula (see \cite{S14}, for details). We later need the area measure $\Psi_k(K,\cdot)$ of $K$, which is the image of $\Xi_k(K;\cdot)$ under the projection $(x,u)\mapsto u$, and the total measure $V_k(K)=\Xi_k(K;\rd\times S^{d-1})$ which is the $k$th intrinsic volume of $K$. The image $\Phi_k(K;\cdot)$ of $\Xi_k(K;\cdot)$ under the other projection $(x,u)\mapsto x$ is usually called the $j$th curvature measure of $K$.

In the following, we prefer a different normalization of these measures, namely we put
$$
C_k(K,\cdot):=\frac{d\kappa_{d-k}}{\binom{d}{k}}\,\Xi_k(K,\cdot)
$$
and
$$
S_k(K,\cdot):=\frac{d\kappa_{d-k}}{\binom{d}{k}}\,\Psi_k(K,\cdot) .
$$
Thus, $S_k(K,\cdot)$ is the marginal measure on $S^{d-1}$ of $C_k(K,\cdot)$. Note that here we deviate from the notation used in \cite{S14}, where $C_k(K,\cdot)$ denotes the re-normalized curvature measure $\Phi_k(K;\cdot)$. Instead, we follow the paper \cite{HR}, and other publications in geometric measure theory, and call this re-normalized support measure the $k$th (generalized) curvature measure of $K$. It gives rise to the {\it mixed curvature measures}
\[C_{r_1,\ldots,r_k}(K_1,\ldots,K_k;\cdot)\,,\]
for  $r_1,\ldots,r_k\in\{0,\ldots,d\}$ with $(k-1)d\le r_1+\cdots+r_k\le kd-1$ and convex bodies $K_1,\ldots,K_k\subset\R^d$, which are finite Borel measures on $\R^{kd}\times S^{d-1}$, defined by a local version of \eqref{transk}, that is, if
$h:\R^{kd}\times S^{d-1}\to [0,\infty]$ is an arbitrary nonnegative, Borel measurable function, then
\begin{align}  \label{TIF}
&{\int_{\R^d}\ldots\int_{\R^d}\int
h(x,x-z_2,\ldots,x-z_k,u)\, C_j(\underline{K}(\underline{z});d(x,u))}\,dz_k\ldots dz_2\nonumber \\
&=\sum_{\substack{0\le r_1,\ldots,r_k\le d\\ r_1+\cdots+r_k=(k-1)d+j}}
\int h(x_1,\ldots,x_k,u)\,C_{r_1,\ldots,r_k}(K_1,\ldots,K_k;d(x_1,\ldots,x_k,u)),
\end{align}
for  $j\in\{0,\ldots,d-1\}$ and with  $\underline{K}(\underline{z}):=K_1\cap(K_2+z_2)\cap\cdots\cap(K_k+z_k)$. Formula \eqref{TIF} was proved in \cite{Hug}, see \cite{Rataj97,HR} for an extension to sets with positive reach
and further references. It generalizes the local iterated translation formula for the curvature measures  $\Phi_j(K,\cdot)$ in \cite{SW}. In fact, the mixed (translative) measures $\Phi^{(j)}_{r_1,\dots ,r_k}(K_1,\dots, K_k,\cdot)$ in \cite{SW} are (up to a constant) the images of $C_{r_1,\ldots,r_k}(K_1,\ldots,K_k;\cdot)$ under the projection $(z_1,\dots,z_k,u)\mapsto (z_1,\dots ,z_k)$.

Concerning flag measures of convex bodies, we refer to the survey \cite{HTW} for background information and to \cite{HRW} for the specific measures used here. In the following, we consider the {\it flag manifold}
$$\Forth(d,k):=\{ (u,V)\in S^{d-1}\times G(d,k):\, u\perp V\} ,$$
where $u\perp V$ means that $u$ is orthogonal to the linear subspace $V$. For a convex body $K\subset\rd$ and $k\in\{0,\dots ,d-1\}$, the $k$th {\it flag measure} $\Omega_k(K;\cdot)$ of $K$  is a measure on $\Forth(d,d-1-k)$ defined by
\begin{align*}
&\int g(u,V)\, \Eam_k(K;d(u,V))\\
&\qquad\qquad=\tilde\gamma(d,k)\int_{G(d,\bark)}
\int_{\partial (K|V^\perp)}\sum g(u,V)\,{\cal H}^k(dz)\,\nu_{\bark}^d(dV),
\end{align*}
where $g$ is a bounded measurable function on $F^\perp(d,d-1-k)$ and
the summation is extended over all  exterior unit normal vectors $u\in V^\perp\cap S^{d-1}$
 of $\partial(K|V^\perp)$ at $z$. If $K|V^\perp$ is $(k+1)$-dimensional, then  $u$ is uniquely determined,
for ${\cal H}^k$-almost all $z\in  \partial(K|V^\perp)$ (cf.~\cite[Theorem 2.2.5]{S14}). If $\dim(\partial(K|V^\perp))=k$, then
$u$ is unique up to the sign, for ${\cal H}^k$-almost all $z\in  \partial(K|V^\perp)$.
Finally, if $\dim(\partial(K|V^\perp))<k$, then the inner integral vanishes.
Thus, by \cite[pp.~220-221 and Theorem 4.2.3]{S14}, we obtain
\begin{eqnarray*}
&&\int g(u,V)\, \Eam_k(K;d(u,V))\\
&&\quad =\tilde\gamma(d,k)\int_{G(d,{k+1})}
\int_{S^{d-1}\cap U}g(u,U^\perp)\,S^U_k(K\vert U,du)\,\nu_{{k+1}}^d(dU),
\end{eqnarray*}
where $S_k^U(K\vert U,\cdot)$
is the $k$th area measure of the orthogonal projection of $K$ onto $U$,
with respect to $U$ as the ambient space. Note that this relation holds irrespective
of the dimension of $K|U$. The constant in the previous two formulas is given by
$$\tilde{\gamma}(d,k):=\frac 12 \binom{d-1}{ k}
\frac{\Gamma\left(\frac{d-k}2\right)\Gamma\left(\frac{k+1}2\right)} {\Gamma\left(\frac 12\right)\Gamma\left(\frac d2\right)} .$$

We will need another description of $\Omega_k(K;\cdot)$ for which we refer to a more general result in \cite{HRW} (see also \cite{HHW}, for the case of polytopes from which the general formula can be obtained by approximation),
\begin{eqnarray}
&&\int g(u,V)\, \Omega_k(K;d(u,V) )\nonumber\\
&&\qquad =\gamma(d,k)\int_{\nor (K)}\sum_{|I|=
\bark}\BK_I(K;x,u)\int_{G^{u^\perp}(d-1,d-1-k)} g(u,V)
\label{ECM}\\
&&\qquad \qquad\times\, \langle V,A_I(K;x,u)\rangle^2\,
\nu^{d-1}_{d-1-k}(dV)\,{\cal H}^{d-1}(d(x,u)), \nonumber
\end{eqnarray}
where
$$\gamma(d,k):=\frac{\binom{d-1}{k}}{\omega_{d-k}},$$
$G^{u^\perp}(d-1,j)$ is the Grassmannian of $j$-dimensional linear subspaces of $u^\perp$, and $\nu^{d-1}_{j}$ denotes the Haar probability measure on this space. In the scalar product $\langle V,A_I(K;x,u)\rangle^2$, we interpret $V$ and $A_I(K;x,u)$ as one of the two possible associated elements of the oriented Grassmannian. This representation is
similar to formula \eqref{suppmeas} for the support measures $\Xi_k(K;\cdot)$. The crucial difference is that for each $(x,u)$ in the normal bundle of $K$ and for each $I$, the flag measures involve
an additional averaging  of $g(u,V)\langle V,A_I(K;x,u)\rangle^2$ over the linear subspaces $V\in G^{u^\perp}(d-1,\bark)$; these averages are exactly the weights with which the  products $\BK_I(K;x,u)$ of generalized curvatures have to be multiplied.

From this representation it can be seen that the projection $(u,V)\mapsto u$ maps $\Omega_k(K;\cdot)$ to the $k$th area measure. In fact, we have
$$\Omega_k(K;\cdot\times G(d,{d-1-k}))=\binom{d}{k}(d\kappa_{d-k})^{-1}S_k(K,\cdot).$$

We remark that also the  flag measures $\Omega_k(K;\cdot )$  can be obtained, alternatively, as coefficients  in a Steiner formula for $K$ on the Grassmannian; see, for instance, \cite{HTW}.

We will later use the following simple fact (see \cite[Equation (15)]{HR}).

\begin{Lemma}  \label{fact}
If $L\in G(d,j)$, then $\int_{S^{d-1}}\|u|L\|^p\, \Ha^{d-1}(du)<\infty$ if and only if $p>-j$.
\end{Lemma}

\section{Curvature representation of mixed volumes}\label{3}
As we noted in the introduction, for two convex bodies $K,L$ in $\rd$   the mixed volumes and the mixed translative functionals of $K$ and $L$ satisfy the relation
\begin{equation}\label{laterefs}
V_{n,d-n}(K,L)=\binom{d}{ n}V(K[n],-L[d-n]),
\end{equation}
for $n=1,\dots ,d-1$ (the cases $n=0$ and $n=d$ hold trivially).
For $V_{k,l}$ with $k+l=d$, the integral representation
\begin{align}
V_{k,l}(K,L)&=\int_{\nor (K)\times\nor (L)}F_{k,l}(\angle (u,v))\sum_{|I|=d-1-k}\sum_{|J|=d-1-l} \BK_I(K;x,u)\BK_J(L;y,v) \nonumber\\
&\quad\times \left\|A_I(K;x,u)\wedge u\wedge A_J(L;y,v)\wedge v\right\|^2\, {\cal H}^{2d-2}(d(x,u,y,v))
\label{IR}
\end{align}
has been proved in \cite[Theorem~2]{RZ95}.
Here, $F_{k,l}$ is a certain function of the angle $\angle(u,v)\in [0,\pi]$ between the unit vectors $u,v\in S^{d-1}$
and $A_I(K;x,u)$ and $A_J(L;y,v)$ are viewed as multivectors.
An important issue related to the use of the function $F_{k,l}$ is that it becomes unbounded as the angle approaches $\pi$ (that is, for $u$ near $-v$). One may define $F_{k,l}(\pi )=0$, say, since  for $u=-v$ we have
$\left\|A_I(K;x,u)\wedge u\wedge A_J(L;y,v)\wedge v\right\|=0$ in \eqref{IR}, but the unboundedness remains and this is the reason why the flag representation
requires certain restrictions on the relative position of the bodies involved (see, for instance, Theorem 2 in \cite{HRW}).

Of course, the representation \eqref{IR} yields a corresponding result for the mixed volume $V(K[k],-L[l])$. For the mixed translative functionals $V_{r_1,\dots ,r_k}(K_1,\dots, K_k)$, a representation generalizing \eqref{IR} has been  established in \cite{HR}, but as we explained in the introduction, this does not imply a corresponding formula for the mixed volume $V(K_1[n_1],\ldots,K_k[n_k])$. We now provide such a curvature representation of mixed volumes for general convex bodies.

For $k\ge 2$, let $K_1,\ldots,K_k\subset\R^d$ be convex bodies and let $k_{ij}=k_{ij}(x_i,u_i), a_{ij}=a_{ij}(x_i,u_i)$, $j=1,\dots,d-1$, be the principal curvatures and principal directions of curvature of $K_i$ at $(x_i,u_i)\in\nor (K_i)$, $i=1,\dots,k$. Given $n=(n_1,\dots,n_k)\in\{0,\dots,d-1\}^k$ with $n_1+\dots+n_k=d$ and $u_1,\dots,u_k\in S^{d-1}$, we put
$F_{n}(u_1,\dots,u_k):=0$ if $u_1=\dots =u_k$, and
\begin{align*}
&F_{n}(u_1,\dots,u_k)\\&\ :=\frac{k^{(k-2)\frac{d}{2}}}{\omega_{(k-1)d}}\int_{S^{k-1}_+}\left(\prod_{i=1}^kt_i^{d-1-n_i}\right)
\left( \sum_{1\leq i<j\leq k}\|t_iu_i-t_ju_j\|^2\right)^{-(k-1)\frac d2}\Ha^{k-1}(dt)
\end{align*}
otherwise.

\begin{Theorem}\label{mivol}
Let $k,d\geq 2$, $n=(n_1,\dots,n_k)\in\{0,\dots,d-1\}^k$ with $n_1+\dots+n_k=d$ and convex bodies $K_1,\dots,K_k\subset\R^d$ be given. Then
\begin{align*}
&\binom{d}{{n_1\dots n_k}} V(K_1[n_1],\dots,K_k[n_k])\\
&\qquad=\int_{\nor (K_1)\times\dots\times\nor (K_k)}F_{n}(u_1,\dots,u_k)\sum_{|I_1|=n_1,\dots,|I_k|=n_k}\left(\prod_{i=1}^k\bK_{I_i^c}(K_i;x_i,u_i)\right)\\
&\qquad\qquad\times\left\|\bigwedge_{i=1}^k A_{I_i}(K_i;x_i,u_i)\right\|^2\, \Ha^{k(d-1)}(d(x_1,u_1,\dots,x_k,u_k)),
\end{align*}
where the sum extends over all subsets $I_i\subset\{1,\dots,d-1\}$ of the prescribed cardinalities.
\end{Theorem}

\begin{proof} We follow an idea of Schneider \cite{S94} and represent $V_d(K_1+\dots +K_k)$ as the volume of a projection from $\R^{dk}$ onto the diagonal space. To be more precise, let
$\ul{K}:=K_1\times\dots\times K_k$ (which is a convex body in $\R^{kd}$). We shall use underlined symbols for points of $\R^{kd}$, such as
$$\ul{x}=(x_1,\dots,x_k),\,\ul{u}=(u_1,\dots,u_k)\in\R^{kd}.$$
Let further
$$L:=\{(x,\dots,x):x\in\R^d\}\in G(kd,d)$$
denote the $d$-dimensional diagonal subspace of $\R^{kd}$. The orthogonal projection to $L$ acts as
$$\ul{x}|L=\frac 1k\left(\sum_{i=1}^kx_i,\dots,\sum_{i=1}^kx_i\right),$$
and it is not difficult to verify that
\begin{equation}  \label{proj_L}
\|\ul{x}|L^\perp\|^2=\frac 1k\sum_{1\leq i<j\leq k}\|x_i-x_j\|^2.
\end{equation}
Since $x\mapsto k^{-1/2}(x,\dots,x)$ is clearly an isometry $\R^d\to L$, we have
$$V_d(K_1+\dots +K_k)=k^{d/2}\Ha^d(\ul{K}|L),$$
and an application of the projection formula \cite[Lemma~4.1]{GHHRW} yields
\begin{equation}
V_d(K_1+\dots +K_k)=\frac{k^{d/2}}{\omega_{(k-1)d}}\int_{\nor(\ul{K})}H(\ul{x},\ul{u})\, \Ha^{kd-1}(d(\ul{x},\ul{u})),
\end{equation}
where
\begin{equation}  \label{Def_H}
H(\ul{x},\ul{u}):=\|\ul{u}|L^\perp\|^{(1-k)d}\sum_{|I|=d}\bK_{I^c}(\ul{K};\ul{x},\ul{u})\langle A_I(\ul{K};\ul{x},\ul{u}),L\rangle^2.
\end{equation}
The unit normal bundle of $\ul{K}$ can be represented as
$$\nor(\ul{K})=\{(\ul{x},\ul{u})\in\R^{kd}\times S^{kd-1}:\, u_i\in\Nor(K_i,x_i), x_i\in K_i,\, i=1,\dots,k\}.$$
We consider first the subsets
$$\nor_i(\ul{K}):=\nor(\ul{K})\cap\{(\ul{x},\ul{u}):\, x_i\in\inte K_i\},\quad i=1,\dots,k.$$
Choosing $i=1$ for simplicity, we get that $\nor_1(\ul{K})$ and $\inte K_1\times\nor(K_2\times\dots\times K_k)$ are isometric, and at any $(\ul{x},\ul{u})\in\nor_1(\ul{K})$ there are $d$ principal directions $(e_j,0,\dots,0)$, $j=1,\dots,d$, with vanishing principal curvatures. Thus, the sum in \eqref{Def_H} reduces to one summand, and since
$$\langle (e_1,0,\dots,0)\wedge\dots\wedge(e_d,0,\dots,0),L\rangle^2=\frac{1}{k^d},$$
we get, for $(\ul{x},\ul{u})\in \nor_1(\ul{K})$ (which implies $u_1=0$),
$$H(\ul{x},\ul{u})=\frac{1}{k^d}\,\|\ul{u}|L^\perp\|^{(1-k)d}
\bK_{I_0}(K_2\times\dots\times K_k;(x_2,\dots,x_k),(u_2,\dots,u_k))$$
with $I_0=\{1,\dots,(k-1)d-1\}$. Hence,
$$
\frac{k^{d/2}}{\omega_{(k-1)d}}\int_{\nor_1(\ul{K})}H\, d\Ha^{kd-1}=V_d(K_1)\psi(K_2,\dots,K_k)
$$
with some function $\psi$ independent of $K_1$ and homogeneous of degree $0$ in $K_2,\dots ,K_k$. Since $V_d(K_1+\dots +K_k)$ can be expanded as a sum of functionals of specific homogeneity degrees (see \eqref{polexp}) and the only term which is $d$-homogeneous in $K_1$ is $V_d(K_1)$, we get
\begin{equation}\label{eqeq1}
\frac{k^{d/2}}{\omega_{(k-1)d}}\int_{\nor_i(\ul{K})}H\, d\Ha^{kd-1}=V_d(K_i),
\end{equation}
first for $i=1$, but then similarly for all $i=1,\dots,k$.
Also, $H=0$ on $\nor_i (\ul{K})\cap\nor_j(\ul{K})$ if $i\neq j$ (this follows from the above argument since then at least $2d$ principal curvatures vanish). Thus we obtain
\begin{equation}
V_d(K_1+\dots +K_k)=\sum_{i=1}^kV_d(K_i)+\frac{k^{d/2}}{\omega_{(k-1)d}}\int_{\nor_*(\ul{K})}H(\ul{x},\ul{u})\, \Ha^{kd-1}(d(\ul{x},\ul{u}))
\end{equation}
with
$$\nor_*(\ul{K}):=\nor(\ul{K})\cap\{(\ul{x},\ul{u}):\, x_i\in\partial K_i,\, i=1\dots,k\}.$$
Consider
$$f:(\nor (K_1)\times\dots\times\nor (K_k))\times S^{k-1}_+\to\nor_*(\ul{K}),\quad ((\ul{x,u}),t)\mapsto(\ul{x},\ul{tu}),$$
where $(\ul{x,u}):= (x_1,u_1,\dots ,x_k,u_k)$ and $\ul{tu}:=(t_1u_1,\dots,t_ku_k)$. The mapping $f$ is clearly Lipschitz, injective and we have $\Ha^{kd-1}(\nor_*(\ul{K})\setminus\im f)=0$. Then the coarea formula yields
$$\int_{(\nor (K_1)\times\dots\times\nor (K_k))\times S^{k-1}_+}(H\circ f)\ap J_{kd-1}f\, d\Ha^{kd-1}=\int_{\nor_*(\ul{K})}H\, d\Ha^{kd-1}.$$
In order to obtain the approximate Jacobian $\ap J_{kd-1}f$ of $f$ at $((\ul{x,u}),t)$, for almost all $((\ul{x,u}),t)\in
(\nor (K_1)\times\dots\times\nor (K_k))\times S^{k-1}_+$, let $(v_1,\dots,v_{k-1},t)$ be an orthonormal basis of $\R^k$ and
put $k_{ij}:=k_j(K_i;x_i,u_i)$ and $a_{ij}:=a_j(K_i;x_i,u_i)$, for $i\in\{1,\ldots,k\}$, $j\in \{1,\ldots,d-1\}$. Then, the vectors
$$\frac{1}{\sqrt{1+k_{ij}^2}}(\underbrace{0,\dots,0}_{2(i-1)},a_{ij},k_{ij}a_{ij},\underbrace{0,\dots,0}_{2(k-i)},0),\quad i=1,\dots, k,\,j=1,\dots,d-1,$$
and
$$(\underbrace{0,\dots,0}_{2k},v_i),\quad i=1,\dots,k-1,$$
form an orthonormal basis of $\Tan^{kd-1}((\nor (K_1)\times\dots\times\nor (K_k))\times S^{k-1}_+,((\ul{x,u}),t))$, for almost all $((\ul{x,u}),t)$, and these vectors are mapped by the approximate differential $\ap Df((\ul{x,u}),t)$ onto the vectors
$$\frac{1}{\sqrt{1+k_{ij}^2}}(\underbrace{0,\dots,0}_{i-1},a_{ij},\underbrace{0,\dots,0}_{k-1},t_ik_{ij}a_{ij},\underbrace{0,\dots,0}_{k-i}),\quad i=1,\dots, k,\,j=1,\dots,d-1,$$
and
$$(0,\dots,0,v_1^iu_1,\dots,v_k^iu_k),\quad i=1,\dots, k-1.$$
These vectors are again orthogonal and it follows that
$$\ap J_{kd-1}f((\ul{x,u}),t)=\prod_{i=1}^k\prod_{j=1}^{d-1}\frac{\sqrt{1+t_i^2k_{ij}^2}}{\sqrt{1+k_{ij}^2}}.$$
We also see that the vectors
$$b_{ij}:=\begin{cases}
(\underbrace{0,\dots,0}_{i-1},a_{ij},\underbrace{0,\dots,0}_{k-i}),&i=1,\dots,k,\, j=1,\dots,d-1,\\
(v_1^iu_1,\dots,v_k^iu_k),&i=1,\dots,k-1,\, j=d,
\end{cases}$$
are generalized principal directions of curvature of $\ul{K}$ at $f((\ul{x,u}),t)$ with corresponding principal curvatures $t_ik_{ij}$, if $1\leq j\leq d-1$, and $\infty$, if $j=d$.
Thus, in \eqref{Def_H} we may omit the index sets $I\subset\{1,\dots,kd\}$ which, written as subsets of $\{1,\dots ,k\}\times\{1,\dots ,d\}$ (according to the consideration above), contain an index $(i,d)$. With respect to this product form, let $I=I_1\cup\dots\cup I_k$ be an index set of cardinality $d$ decomposed into subsets $I_i$ corresponding to indices from $\{i\}\times\{1,\dots,d-1\}$. Then, we can write
$$\bK_{I^c}(\ul{K};f((\ul{x,u}),t))=\prod_{i=1}^kt_i^{d-1-|I_i|}\prod_{j\in I_i^c}k_{ij}/\prod_{j=1}^{d-1}\sqrt{1+t_i^2k_{ij}^2}$$
and
$$\langle A_I(\ul{K};f((\ul{x,u}),t)),L\rangle^2=k^{-d}\left\|\bigwedge_{i=1}^k\bigwedge_{j\in I_i}a_{ij}\right\|^2,$$
hence,
\begin{align*}
\int_{\nor_*(\ul{K})}H\, d\Ha^{kd-1}&=k^{-d}\int_{(\nor (K_1)\times\dots\times\nor (K_k))\times S^{k-1}_+}\|\ul{tu}|L^\perp\|^{(1-k)d}\\
&\qquad \times\sum_{|I_1|+\dots+|I_k|=d}\left(\prod_{i=1}^kt_i^{d-1-|I_i|}\bK_{I_i^c}(K_i;x_i,u_i)\right) \\ &\qquad \times \left\|\bigwedge_{i=1}^k\bigwedge_{j\in I_i}a_{ij}\right\|^2\Ha^{kd-1}(d((\ul{x,u}),t)).
\end{align*}
In the above integral, the summands with $|I_i|=n_i$ produce integrals with homogeneity degree $n_i$ in $K_i$.
To verify this, let $n_i\in\{0,\ldots,d-1\}$, $n_1+\ldots+n_k=d$, and let $\lambda_i>0$, for $i=1,\ldots,k$.
Let $g:(S^{d-1})^k\times S^{d-1}_+\to [0,\infty)$ be measurable and let
\begin{align*}
\Upsilon_{\underline{n}}(K_1,\ldots,K_k):&=
\int_{(\nor (K_1)\times\dots\times\nor (K_k))\times S^{k-1}_+}g(\ul{u},t) \sum_{|I_1|=n_1,\ldots,|I_k|=n_k}
\\
&\qquad\times \left(\prod_{i=1}^k \bK_{I_i^c}(K_i;x_i,u_i)\right)  \left\|\bigwedge_{i=1}^k\bigwedge_{j\in I_i}a_{ij}\right\|^2\Ha^{kd-1}(d((\ul{x,u}),t)).
\end{align*}
Consider the map
\begin{align*}
F_{\ul{\lambda}} :(\nor (K_1)\times\dots\times\nor (K_k))\times S^{k-1}_+&\to (\nor (\lambda_1K_1)\times\dots\times\nor (\lambda_kK_k))\times S^{k-1}_+,\\
((\ul{x,u}),t)&\mapsto (\lambda_1x_1,u_1,\ldots,\lambda_kx_k,u_k,t).
\end{align*}
For $\mathcal{H}^{d-1}$-almost all $(x_i,u_i)\in\nor(K_i)$, we have
\begin{align*}
k_{ij}(\lambda_i K_i;\lambda_ix_i ,u_i)&=\lambda_i^{-1}k_{ij}(K_i;x_i,u_i),\\
a_{ij}(\lambda_i K_i;\lambda_ix_i ,u_i)&=a_{ij}(K_i;x_i,u_i),
\end{align*}
for $i\in\{1,\ldots,k\}$ and $j\in \{1,\ldots,d-1\}$. Moreover, for $\mathcal{H}^{kd-1}$-almost all $((\ul{x,u}),t)\in
(\nor (K_1)\times\dots\times\nor (K_k))\times S^{k-1}_+$, we have
$$
\ap J_{kd-1}F_{\ul{\lambda}}((\ul{x,u}),t)=\prod_{i=1}^k\lambda_i^{d-1}\prod_{j=1}^{d-1}\frac{\sqrt{1+\left
(\lambda_i^{-1}k_{ij}(K_i;x_i,u_i)\right)^2}}{\sqrt{1+k_{ij}(K_i;x_i,u_i)^2}}.
$$
Now we get
\begin{align*}
&\Upsilon_{\underline{n}}(\lambda_1K_1,\ldots,\lambda_kK_k)\\
&=
\int_{(\nor (\lambda_1K_1)\times\dots\times\nor (\lambda_kK_k))\times S^{k-1}_+}g(\ul{u},t)\sum_{|I_1|=n_1,\ldots,|I_k|=n_k}\\
&\qquad \times\left(\prod_{i=1}^k \bK_{I_i^c}(\lambda_iK_i;y_i,u_i)\right)  \left\|\bigwedge_{i=1}^k\bigwedge_{j\in I_i}a_{ij}(\lambda_iK_i;y_i,u_i)\right\|^2\Ha^{kd-1}(d((\ul{y,u}),t))\\
&=\int_{(\nor ( K_1)\times\dots\times\nor ( K_k))\times S^{k-1}_+}g(\ul{u},t)\sum_{|I_1|=n_1,\ldots,|I_k|=n_k}\\
&\qquad \times\prod_{i=1}^k \left[\left(\lambda_i^{-(d-1-n_i)} \bK_{I_i^c}(K_i;x_i,u_i)\right)
\prod_{j=1}^{d-1}\frac{\sqrt{1+k_{ij}(K_i;x_i,u_i)^2}}{\sqrt{1+\left
(\lambda_i^{-1}k_{ij}(K_i;x_i,u_i)\right)^2}}\right]
\\
&\qquad\times \left\|\bigwedge_{i=1}^k\bigwedge_{j\in I_i}a_{ij}(K_i;x_i,u_i)\right\|^2
\ap J_{kd-1}F_{\ul{\lambda}}((\ul{x,u}),t)\Ha^{kd-1}(d((\ul{x,u}),t))\\
&=\lambda_{1}^{n_1}\cdots \lambda_{k}^{n_k}\Upsilon_{\underline{n}}( K_1,\ldots, K_k).
\end{align*}
Thus, the expansion \eqref{polexp} of $V_d(K_1+\dots+K_k)$ and, finally, a use of \eqref{proj_L} complete the proof.
\end{proof}

\bigskip

\noindent
{\bf Remark.} Relation \eqref{eqeq1} can also be obtained directly, without using the fact that mixed volumes have  a
polynomial expansion. We show this for $i=1$. First, we observe that if $L\in G(p,d)$ and $\beta>0$, then
\begin{align*}
&\int_{S^{p-1}}(1+\beta\|u|L^\perp\|^2)^{-\frac{p}{2}}\, \mathcal{H}^{p-1}(du)\\
&\qquad=\int_{S^{p-1}\cap L}\int_{S^{p-1}\cap L^\perp}\int_0^{\frac{\pi}{2}}
(1+\beta\cos^2t)^{-\frac{p}{2}}(\cos t)^{p-d-1}(\sin t)^{d-1}\, \\
&\qquad\qquad\qquad \times dt\,\mathcal{H}^{p-d-1}(dy)\,\mathcal{H}^{d-1}(dx)\\
&\qquad=\omega_d\omega_{p-d}\int_0^1(1+\beta r^2)^{-\frac{p}{2}}(1-r^2)^{\frac{d-2}{2}}r^{p-d-1}\, dr\\
&\qquad=\frac{1}{2}\omega_d\omega_{p-d}\int_0^1s^{\frac{p-d-2}{2}}(1-s)^{\frac{d-2}{2}}(1+\beta s)^{-\frac{p}{2}}\, ds\\
&\qquad=\frac{1}{2}\omega_d\omega_{p-d}B\left(\tfrac{p-d}{2},\tfrac{d}{2}\right)
{_2F_1}\left(\tfrac{p}{2},\tfrac{p-d}{2};\tfrac{p}{2};-\beta\right)\\
&\qquad=\omega_p (1+\beta)^{-\frac{p-d}2},
\end{align*}
where we used in the first equality that the (smooth) map  $F:(S^{p-1}\cap L)\times (S^{p-1}\cap L^\perp)\times (0,\pi/2)\to S^{p-1}$,
$(x,y,t)\mapsto \sin(t) x+\cos(t) y$, parameterizes $S^{p-1}$ (up to a set of measure zero) and
has the (approximate) Jacobian $J_{p-1}F(x,y,t)=(\sin t)^{d-1}(\cos t)^{p-d-1}$, and in the last two equalities some basic properties of Gaussian hypergeometric functions (see \cite[Equations~15.1.1, 15.1.8, 15.3.1]{AS64}).

Now we put $\tilde x=(x_2,\ldots,x_k)$, use that $u_1=0$ for $(\underline{x},\underline{u})\in \nor_1(\underline{K})$,
and hence
$$
\|\underline{u}|L^\perp\|^2 =\frac{1}{k}\left(\sum_{i\ge 2}\|u_i\|^2+\sum_{2\le i<j\le k}\|u_i-u_j\|^2\right)
=\frac{1}{k}\left(1+(k-1)\|\underline{\tilde u}|\tilde{L}^\perp\|^2\right),
$$
where $\tilde{L}=\{(x,\ldots,x)\in\R^{(k-1)d}:x\in\R^d\}\in G((k-1)d,d)$. Then the coarea formula,
applied with the map $(\tilde x,\tilde u)\mapsto \tilde u$ (with an approximate Jacobian as in  \eqref{jacobp2}), and the above equality with $p=(k-1)d$ and $\beta=k-1$, yield that
\begin{align}
&\frac{k^{\frac{d}{2}}}{\omega_{(k-1)d}}\int_{\nor_1(\ul{K})}H\, d\Ha^{kd-1}\nonumber\\
&\quad= \frac{k^{-\frac{d}{2}}}{\omega_{(k-1)d}}V_d(K_1)\int_{\nor(\tilde{\underline{K}})}
\|\underline{u}|L^\perp\|^{(1-k)d} \mathbb{K}_{I_0}(\tilde{\underline{K}};\underline{\tilde x},\underline{\tilde u})\,
 \Ha^{(k-1)d-1}(d(\underline{\tilde x},\underline{\tilde u})) \nonumber\\
&\quad=\frac{V_d(K_1)}{\omega_{(k-1)d}}k^{\frac{(k-2)d}{2}}
\int_{S^{(k-1)d-1}}\left(1+(k-1)\|\underline{\tilde u}|\tilde{L}^\perp\|^2\right) ^{\frac{(1-k)d}{2}}\,
\mathcal{H}^{(k-1)d-1}(d\underline{\tilde u})\nonumber\\
&\quad=V_d(K_1),\nonumber
\end{align}
where $\tilde{\underline{K}}:=K_2\times\cdots \times K_k$.
\bigskip

Next we emphasize some special cases of Theorem \ref{mivol}.

\bigskip

\noindent
{\bf Remarks.}
(a) If
  $K_1,\ldots,K_k$ are convex bodies of class $C^{1,1}$ (that is, $\partial K$ is of class $C^1$ and the
	exterior unit normal map (the Weingarten map) is Lipschitz), then the integral representation in Theorem \ref{mivol} simplifies. Namely, we then have
\begin{align*}
&\binom{d}{n_1\ldots n_k}V(K_1[n_1],\ldots,K_k[n_k])\\
&\qquad =\int_{\partial K_1} \cdots \int_{\partial K_k}  F_n( n_{\underline{K}}(\underline{x}))
\sum_{{|I_i|=
n_i \atop i=1,\ldots,k}}\left\{\prod_{i=1}^k\prod_{j\in
I_i^c}k_{ij}(x_i)\right\}\\
&\qquad\qquad \times\left\|\bigwedge_{i=1}^k\bigwedge_{j\in
I_i}a_{ij}(x_i)\right\|^2\, \mathcal{H}^{d-1} (d x_k)\cdots \mathcal{H}^{d-1} (d x_1)
\,,
\end{align*}
where
 $n_{\underline{K}}(\underline{x})=(n_{K_1}(x_1),\ldots,
n_{K_k}(x_k))$ and $n_{K_i}(x_i)$ is the unique exterior unit normal vector of $K_i$ at $x_i\in\partial K_i$. Furthermore,
 $k_{ij}(x_i)$, $j=1,\ldots,d-1$, are the principal curvatures of $K_i$ at $x_i$ with corresponding eigenvectors $a_{ij}(x_i)$, $j=1,\ldots,d-1$, of the (generalized) Weingarten map, for $i=1,\ldots,k$. Here we use that if $K$ is of class
$C^{1,1}$, then for $\mathcal{H}^{d-1}$-almost all $(x,u)\in \nor(K)$ we have $k_i(x,u)=k_i(x)\in[0,\infty)$ for $i=1,\ldots,d-1$ (see \cite[Lemma 3.1]{Hug96}). Moreover, the projection map $\pi_1:\nor(K)\to\partial K$, $(x,u)\mapsto x$, has the approximate Jacobian
$$
\ap J_{d-1}\pi_1(x,u)=\prod_{i=1}^{d-1}\frac{1}{\sqrt{1+k_i(K;x,u)^2}}.
$$

(b)  If the convex bodies
 $K_1,\ldots,K_k$ have support functions of class $C^{1,1}$ (that is, the differential exists and is a
 Lipschitz map), then  $K_1,\ldots,K_k$ are strictly convex.
See Lemma 1 in \cite{HRW} for equivalent conditions on a convex body to have a support function of class $C^{1,1}$.
In this case, we obtain
\begin{align*}
&\binom{d}{n_1\ldots n_k}V(K_1[n_1],\ldots,K_k[n_k])\\
&\qquad=\int_{S^{d-1}}\cdots \int_{S^{d-1}} F_n(\underline{u})
\sum_{|I_i|=n_i\atop i=1,\ldots,k}
\left\{\prod_{i=1}^k\prod_{j\in
I_i}r_{ij}(u_i)\right\}\\
&\qquad\qquad\times\left\|\bigwedge_{i=1}^k\bigwedge_{j\in
I_i}a_{ij}(u_i)\right\|^2
\, \mathcal{H}^{d-1} (d u_k)\cdots \mathcal{H}^{d-1} (d u_1)\,,
\end{align*}
where $r_{ij}(u_i)$, $j=1,\ldots,d-1$, are the radii of curvature of $K_i$ in direction $u_i$ with corresponding eigenvectors $a_{ij}(u_i)$, $j=1,\ldots,d-1$, of the (generalized) reverse Weingarten map, for $i=1,\ldots,k$.

Here we use the fact that if the support function $h_K$ of $K$ is of class
$C^{1,1}$, then for $\mathcal{H}^{d-1}$-almost all $(x,u)\in \nor(K)$ we have $k_i(x,u)^{-1}=r_i(x)\in[0,\infty)$ and $k_i(x,u)>0$ for $i=1,\ldots,d-1$ (see \cite[Lemma 3.4]{Hug96}).  Moreover,
the map $S^{d-1}\to \nor(K)$, $u\mapsto (x,u)$, is Lipschitz and the projection map $\pi_2:\nor(K)\to S^{d-1}$, $(x,u)\mapsto u$, has the approximate Jacobian
\begin{equation}\label{jacobp2}
\ap J_{d-1}\pi_2(x,u)=\prod_{i=1}^{d-1}\frac{k_i(K;x,u)}{\sqrt{1+k_i(K;x,u)^2}}.
\end{equation}

(c) The important special case where $K_1,\ldots,K_k$ are convex polytopes will be treated in the next
section.

\section{Mixed volumes of polytopes}\label{4}

Let $P_1,\ldots,P_k$ be polytopes in $\R^d$,
and let $k\ge 2$ and $n=(n_1,\ldots,n_k)$ be as in Theorem \ref{mivol}.
For a polytope $P$ in $\R^d$, we write $\mathcal{F}_j(P)$
for the set of $j$-dimensional faces of $P$, and let $N(P,F)$ denote the
normal cone of $P$ at $F\in \mathcal{F}_j(P)$. Furthermore, we put $n(P,F):=N(P,F)\cap S^{d-1}$. Then, Theorem \ref{mivol} implies
\begin{align}\label{reppol1}
&\binom{d}{n_1\ldots n_k}V(P_1[n_1],\ldots,P_k[n_k])\nonumber\\
&\qquad=\sum_{F_1\in \mathcal{F}_{n_1}(P_1)}\ldots\sum_{F_k\in \mathcal{F}_{n_k}(P_k)}
[ F_1,\ldots,F_k ]^2 V_{n_1}(F_1)\cdots V_{n_k}(F_k) \\
&\qquad\qquad\times\int_{n(P_1,F_1)}\cdots
\int_{n(P_k,F_k)}F_n(\underline{u})\,
\mathcal{H}^{d-1-n_k}(du_k)\cdots\mathcal{H}^{d-1-n_1}(du_1),\nonumber
\end{align}
where  $[F_1,\ldots,F_k]$
denotes the $d$-dimensional volume of the parallelepiped which is obtained as the sum of the unit cubes in
the affine hulls of $F_1,\ldots,F_k$, respectively. In fact, for a polytope
$P\subset\R^d$ we have the disjoint decomposition
$$
\nor(P)=\bigcup_{n=0}^{d-1}\bigcup_{F\in\mathcal{F}_n(P)}\text{relint}(F)\times n(P,F) ,
$$
and for $\mathcal{H}^{d-1}$-almost all $(x,u)\in\nor(P)$ with $x\in \text{relint}(F)$,
$u\in n(P,F)$ and $F\in\mathcal{F}_n(P)$ precisely $n$ of the curvatures $k_i(x,u)$ are zero and the remaining
$d-1-n$ of the curvatures $k_i(x,u)$ are infinite. Moreover, $a_i(x,u)$ is in the linear space parallel to $F$ precisely
if $k_i(x,u)=0$. Formula \eqref{reppol1} now follows by arguing as in \cite[p.~1542]{HS14}.

This special case of Theorem \ref{mivol} is related to
  \cite[Theorem 4.1]{S94},
see also \cite[p.~311]{S14}, as explained below.   The latter result
describes a method of computing
$V(P_1[n_1],\ldots,P_k[n_k])$  by summing the weighted volumes
$[ F_1,\ldots,F_k ]  V_{n_1}(F_1)\cdots V_{n_k}(F_k) $, where
the faces $F_i\in \mathcal{F}_{n_i}(P_i)$ for $i=1,\ldots,k$ are chosen subject to
 a selection rule. More explicitly,
\begin{align}
&\binom{d}{n_1\ldots n_k}V(P_1[n_1],\ldots,P_k[n_k])\nonumber\\
&\qquad=\sum_{F_1\in \mathcal{F}_{n_1}(P_1)}{\stackrel{{\ast}}\ldots}\sum_{F_k\in \mathcal{F}_{n_k}(P_k)}
[F_1,\ldots,F_k]V_{n_1}(F_1)\cdots V_{n_k}(F_k),\label{reppol2}
\end{align}
where the star indicates that the summation extends over all $k$-tuples of faces $(F_1,\ldots,F_k)\in \mathcal{F}_{n_1}(P_1)
\times \cdots\times \mathcal{F}_{n_k}(P_k)$ for which $\text{dim}(F_1+\cdots+F_k)=d$ and
\begin{equation}\label{nonempt}
\bigcap_{i=1}^k[N(P_i,F_i)-x_i]\neq \emptyset.
\end{equation}
Here $x_1,\ldots,x_k\in\R^d$ are fixed vectors (not all zero) such that
$$
\bigcap_{i=1}^k[\text{relint}(N(P_i,G_i))-x_i]= \emptyset
$$
whenever $G_i\in\mathcal{F}(P_i)$ and $\text{dim}(G_1)+\cdots+\text{dim}(G_k)>d$.
Any such $k$-tuple of vectors
$(x_1,\ldots,x_k)\in\R^{kd}$ is called {admissible} for the given polytopes.

Let $(x_1,\ldots,x_k)\in L^\perp\cap S^{kd-1}$. Hence $x_1+\cdots+x_k=0$ and not all of the vectors are zero.
Then, we have $(L+\underline{x})\cap N(\underline{P},\underline{F})\neq \emptyset$ if and only if \eqref{nonempt}
is satisfied. Thus we obtain
\begin{align*}
&\binom{d}{n_1\ldots n_k}V(P_1[n_1],\ldots,P_k[n_k])\\
&\qquad=\sum_{F_1\in \mathcal{F}_{n_1}(P_1)}\ldots\sum_{F_k\in \mathcal{F}_{n_k}(P_k)}
[F_1,\ldots,F_k]V_{n_1}(F_1)\cdots V_{n_k}(F_k)\\
&\qquad\qquad \times\mathbf{1}\{(L+\underline{x})\cap N(\underline{P},\underline{F})\neq \emptyset\},
\end{align*}
provided that $(x_1,\ldots,x_k)$ is admissible. It follows from the argument in \cite[p.~409]{GritzmannKlee}
that $\mathcal{H}^{(k-1)d-1}$-almost all $\underline{x}\in L^\perp\cap S^{kd-1}$ are
admissible. For the proof, let $\underline{x}\in L^\perp$ be not admissible (for the given polytopes). Then there are faces $G_i\in\mathcal{F}(P_i)$ with  $\text{dim}(G_1)+\cdots+\text{dim}(G_k)>d$ and such that
$\cap_{i=1}^k[\text{relint}(N(P_i,G_i))-x_i]\neq \emptyset$. Hence there is some $z\in \R^d$ such that
$$
\underline{z}+\underline{x}\in N(\underline{P},\underline{G})=N(P_1,G_1)\times \cdots \times N(P_k,G_k),
$$
where $\underline{z}=(z,\ldots,z)$. Since $\underline{x}$ is not admissible if and only if
$\underline{z}+\underline{x}$ is not admissible for all $z\in\R^d$, we get
$L+\underline{x}\subset N(\underline{P},\underline{G})$
whenever $\underline{x}$ is not admissible.  Now let $N_a$ denote the set of all
$\underline{x}\in L^\perp$ such that $\underline{x}$ is not admissible. Since
$$N_a=(L+ N_a)|L^\perp\subset \bigcup\Lin(N(\underline{P},\underline{G}))|L^\perp,$$
where the union extends over all $k$-tuples of faces  $G_i\in\mathcal{F}(P_i)$ with
 $\text{dim}(G_1)+\cdots+\text{dim}(G_k)>d$,
and
$$
\text{dim}(\Lin(N(\underline{P},\underline{G})))\le \sum_{i=1}^k(d-\text{dim}(G_i))\le kd-(d+1)=(k-1)d-1
$$
for any such $k$-tuple, and since there are only finitely many of such $k$-tuples,
it follows that $\mathcal{H}^{(k-1)d}(N_a)=0$. Since $\underline{x}$ is not admissible if and only
if $\lambda \underline{x}$ is not admissible for all $\lambda>0$, we get
$\mathcal{H}^{(k-1)d}(N_a\cap S^{kd-1})=0$.

Therefore, integration over $L^\perp\cap S^{kd-1}$ yields
\begin{align*}
&\binom{d}{n_1\ldots n_k}V(P_1[n_1],\ldots,P_k[n_k])\nonumber\\
&\qquad=\frac{1}{\omega_{(k-1)d}}\sum_{F_1\in \mathcal{F}_{n_1}(P_1)}\ldots\sum_{F_k\in \mathcal{F}_{n_k}(P_k)}
[F_1,\ldots,F_k]V_{n_1}(F_1)\cdots V_{n_k}(F_k)\nonumber\\
&\qquad\qquad\times \int_{L^\perp\cap S^{kd-1}}\mathbf{1}\{(L+\underline{x})\cap N(\underline{P},\underline{F})\neq \emptyset\}\, \mathcal{H}^{(k-1)d-1}(d\underline{x}).
\end{align*}

In order to provide the connection between the representations \eqref{reppol1} and
\eqref{reppol2}, we now show that
\begin{align}\label{mixedangle}
&[F_1,\ldots,F_k] \int_{n(P_1,F_1)}\ldots
\int_{n(P_k,F_k)}F_n(\underline{u})\,
\mathcal{H}^{d-1-n_1}(du_1)\ldots\mathcal{H}^{d-1-n_k}(du_k)\nonumber\\
&\qquad=\frac{1}{\omega_{(k-1)d}}\int_{L^\perp\cap S^{kd-1}}
\mathbf{1}\{(L+\underline{x})\cap N(\underline{P},\underline{F})\neq
\emptyset\}\, \mathcal{H}^{(k-1)d-1}(d\underline{x}).
\end{align}
We start by recalling  an auxiliary result. Let $L\in G(p,m)$ and $m\in\{1,\ldots,p-1\}$.
Let $P\subset\R^p$ be a polytope and $F\in\mathcal{F}_m(P)$. Then \cite[(33)]{GHHRW} states that
\begin{align}\label{relappl}
&\int_{n(P,F)}[F,L^\perp]\|u|L^\perp\|^{m-p}\, \mathcal{H}^{p-1-m}(du)\nonumber\\
&\qquad =\int_{L^\perp\cap S^{p-1}} \mathbf{1}\{(L+u)\cap N(P,F) \neq \emptyset\}
 \, \mathcal{H}^{p-m-1}(du).
\end{align}
We apply \eqref{relappl} to $\underline{P}=P_1\times\cdots \times P_k\in\R^{kd}$, its face
$\underline{F}=F_1\times\cdots \times F_k\in\mathcal{F}_d(\underline{P})$ and to the linear subspace
$L=\{(x,\ldots,x)\in\R^{kd}:x\in\R^d\}$ with $p=kd$ and $m=d$. Then we get
\begin{align*}
I:&=\int_{L^\perp \cap S^{kd-1}}\mathbf{1}\{(L+\underline{v})\cap N(\underline{P},\underline{F})\neq \emptyset\}
\, \mathcal{H}^{(k-1)d-1}(d\underline{v})\\
&=\int_{n(\underline{P},\underline{F})}[F_1\times \cdots\times F_k,L^\perp] \|\underline{u}|L^\perp\|^{-(k-1)d}
\, \mathcal{H}^{(k-1)d-1}(d\underline{u}).
\end{align*}
The map $G:n(P_1,F_1)\times\cdots\times n(P_k,F_k)\times S^{k-1}_+ \to n(\underline{P},\underline{F})$ given by
$$
G(u_1,\ldots,u_k,t)= (t_1u_1,\ldots,t_ku_k),
$$
is Lipschitz, injective and onto up to a set of measure zero. It is easy to check that the approximate Jacobian of $G$ is
$\ap J_{kd-1}G(\underline{u},t)=t_1^{d-1-n_1}\cdots t_k^{d-1-n_k}$. Moreover, since
$$[F_1\times \cdots\times F_k,L^\perp]=k^{-\frac{d}{2}}[F_1,\ldots,F_k],$$
we get
\begin{align*}
I&=\int_{n(P_1,F_1)}\cdots \int_{n(P_k,F_k)}\int_{S^{k-1}_+}
 [F_1,\ldots,F_k] k^{\frac{(k-2)d}{2}}\Big(\sum_{1\le i<j\le k}\|t_iu_i-t_ju_j\|^2
\Big)^{-\frac{(k-1)d}{2}}\\
&\qquad\qquad \times t_1^{d-1-n_1}\cdots t_k^{d-1-n_k}\, \mathcal{H}^{k-1}(dt)\,
 \mathcal{H}^{d-1-n_k}(du_k)\cdots \mathcal{H}^{d-1-n_1}(du_1)\\
&=\omega_{(k-1)d}\, [F_1,\ldots,F_k] \int_{n(P_1,F_1)}\cdots
\int_{n(P_k,F_k)}F_n(\underline{u})\, \\
&\qquad\qquad \times
\mathcal{H}^{d-1-n_1}(du_1)\cdots\mathcal{H}^{d-1-n_k}(du_k),
\end{align*}
which provides the asserted relation.

Equation \eqref{mixedangle} suggests to define a {\em mixed exterior angle}
of $P_1,\dots, P_k$ at the faces $F_1,\dots ,F_k$ by
\begin{align*}
&\beta (F_1,\dots ,F_k;P_1,\dots , P_k)\\
&\quad:= [F_1,\ldots,F_k] \int_{n(P_1,F_1)}\cdots
\int_{n(P_k,F_k)}F_n(\underline{u})\,
\mathcal{H}^{d-1-n_1}(du_1)\cdots\mathcal{H}^{d-1-n_k}(du_k) .
\end{align*}
This is a number between 0 and 1, and \eqref{reppol1} thus becomes \eqref{mixedvolpolytopes} in the following
corollary.

\begin{Corollary}
Let $k,d\ge 2$, let $P_1,\ldots,P_k$ be polytopes in $\R^d$,
 and let  $n=(n_1,\ldots,n_k)\in\{0,\ldots,d-1\}^k$ with $n_1+\cdots +n_k=d$. Then
\begin{align}
\binom{d}{n_1\ldots n_k}&V(P_1[n_1],\ldots,P_k[n_k])=\sum_{F_1\in \mathcal{F}_{n_1}(P_1)}\ldots\sum_{F_k\in \mathcal{F}_{n_k}(P_k)}\nonumber\\
& \beta (F_1,\dots ,F_k;P_1,\dots , P_k)[ F_1,\ldots,F_k ] V_{n_1}(F_1)\cdots V_{n_k}(F_k) .\label{mixedvolpolytopes}
\end{align}
\end{Corollary}

\medskip

For $k=2$, we have $\beta (F_1,F_2;P_1,P_2)= \gamma (F_1,-F_2,P_1,-P_2)$, where the latter is the common external angle defined in \cite[p. 240]{S14}, hence  \eqref{mixedvolpolytopes} yields a generalization of \cite[(5.66)]{S14} to more than two bodies. For another extension, to mixed measures of translative integral geometry, see Corollary 1 in \cite{HR}.

\section{Flag representation of mixed volumes}\label{5}
The principal aim in this section is to establish a flag representation of  mixed volumes $V(K_1[n_1],\ldots,K_k[n_k])$ for convex bodies $K_1,\ldots,K_k$ in $\R^d$ and $n_1,\ldots,n_k\in\{0,\ldots,d-1\}$ with $n_1+\cdots+n_k=d$. As in the case $k=2$, which we explored in \cite{HRW}, a condition of general position is needed.
We shall show, that this is satisfied, for example, if $k-1$ of the bodies $K_i$ are randomly (and independently) rotated and/or reflected, where a random rotation and/or reflection refers to the (unique) invariant probability measure $\nu_d$ on the orthogonal group $O(d)$.

As a second case, we show that the result holds if the support functions of all but one of the convex bodies $K_i$
are of class $C^{1,1}$ (differentiable and the gradient is a 1-Lipschitz map). As remarked before, the corresponding convex bodies are strictly convex, and in fact, they are freely rolling inside some ball
(see Lemma 1 in \cite{HRW}).

A third condition which ensures the result is that $K_1,\dots ,K_k$ are
convex polytopes in general $(n_1,\dots ,n_k)$-position. To define this notion, recall that $\mathcal{F}_j(K)$
denotes the set of $j$-dimensional faces of a convex polytope $K$, and  $N(K,F)$ is the
normal cone of $F\in \mathcal{F}_j(K)$ at $K$.
Then we say that convex polytopes $K_1,\dots ,K_k\subset\rd$
are in general $(n_1,\dots ,n_k)$-position if $$\bigcap_{i=1}^k N(K_i,F_i)=\{ 0\}$$ holds
for all faces $F_i\in \mathcal{F}_{n_i}(K_i)$, $i=1,\dots ,k$. Note that this condition is satisfied, for instance, if
$$
\Lin(F_1)+\cdots+\Lin(F_k)=\R^d
$$
for all faces $F_i\in \mathcal{F}_{n_i}(K_i)$, $i=1,\dots ,k$. For $k=2$, this latter condition was used in \cite{HRW}. If ${\cal P}$ denotes the set of polytopes in ${\cal K}$, then it is easy to see that the tuples $(K_1,\ldots ,K_k)$ of convex polytopes in general $(n_1,\dots ,n_k)$-position are dense in ${\cal P}^k$ in the Hausdorff metric.

A major step in proving  a flag representation of mixed volumes under any of these assumptions consists in establishing a corresponding flag representation for approximate mixed volumes (of arbitrary convex bodies), which we define next, and then using an approximation argument. For $\varepsilon>0$, $n_1,\ldots,n_k\in\{0,\ldots,d-1\}$ with $n_1+\cdots+n_k=d$ and convex bodies $K_1,\ldots,K_k$ in $\R^d$,
a bounded $\varepsilon$-approximation of $V(K_1[n_1],\ldots,K_k[n_k])$ is defined by
\begin{align*}
&{\binom{d}{n_1\ldots n_k}}V^{(\varepsilon)}(K_1[n_1],\ldots,K_k[n_k])\\
&\quad :=\int_{\nor (K_1)\times\dots\times\nor (K_k)}F_{n}^{(\varepsilon)}(u_1,\dots,u_k)\sum_{|I_1|=n_1,\dots,|I_k|=n_k}\left(\prod_{i=1}^k\bK_{I_i^c}(K_i;x_i,u_i)\right)\\
&\quad\quad\times\left\|\bigwedge_{i=1}^k A_{I_i}(K_i;x_i,u_i)\right\|^2\, \Ha^{k(d-1)}(d(x_1,u_1,\dots,x_k,u_k)),
\end{align*}
where
$$
F^{(\varepsilon)}_n(u_1,\dots ,u_k) := F_n(u_1,\dots ,u_k) {\bf 1}\{ \|\underline{u}|L^\perp\|\ge \ve\}
$$
(recall that $L$ is the diagonal in $\R^{kd}$ and $\underline{u}=(u_1,\ldots,u_k)$).
It is easy to see that $F^{(\varepsilon)}_n$ is nonnegative and
 bounded from above on $(S^{d-1})^k$. The monotone convergence
theorem and Theorem \ref{mivol} show that
$$
V^{(\varepsilon)}(K_1[n_1],\ldots,K_k[n_k])\nearrow V(K_1[n_1],\ldots,K_k[n_k])
$$
as $\varepsilon\searrow 0$. Our first result provides a flag representation for the approximate mixed volumes.

\begin{Theorem}\label{thmapprox}
Let $K_1, \dots ,K_k\subset\rd$  be convex bodies in $\rd$,
$n=(n_1,\dots ,n_k)$ $\in\{1,\ldots,d-1\}^k$ with $n_1+\dots +n_k=d$ and $\ve >0$.
Then, there is a continuous function $\varphi_{n}$ on $F^\perp(d,d-1-n_1)\times\cdots\times F^\perp(d,{d-1-n_k})$ (independent of $K_1,\dots ,K_k$ and $\ve$) such that
\begin{align}\label{flagrepeps}
&V^{(\varepsilon)}(K_1[n_1],\ldots,K_k[n_k])\ = \int_{F^\perp(d,{d-1-n_k})}\cdots\int_{F^\perp(d,d-1-n_1)}
F_n^{(\varepsilon)}(u_1,\dots ,u_k)\nonumber \\
&\quad \times\varphi_n(u_1,U_1,\dots, u_k,U_k)\, \Omega_{n_1}(K_1;d(u_1,U_1))\cdots \Omega_{n_k}(K_k;d(u_k,U_k)).
\end{align}
\end{Theorem}

In order to obtain a suitable function $\varphi_{n}$, as stated in Theorem \ref{thmapprox}, we have to find a solution for an integral equation on
Grassmannians. This is the subject of the next lemma, which generalizes Proposition 2 in \cite{HRW}. In the following,
we write $a\wedge b$ for the minimum of two integers $a,b$. It will always be clear from the context that this is not a  multivector
in the exterior algebra of vector spaces.

\begin{Lemma}\label{Lem}
Let $u_1,\ldots,u_k\in S^{d-1}$ and $1\leq j_1,\ldots,j_k\leq d-1$ be given so that $j_1+\cdots+j_k=d$. Then there exists a continuous function
$$\Phi_{u_1,\ldots,u_k}:\, G^{u_1^\perp}(d-1,j_1)\times\cdots\times G^{u_k^\perp}(d-1,j_k)\to\R$$
such that for all $A_1\in G^{u_1^\perp}(d-1,j_1),\ldots,A_k\in G^{u_k^\perp}(d-1,j_k)$,
$$
\int\cdots\int\Phi_{u_1,\ldots,u_k}(U_1,\ldots,U_k)\, \langle U_1,A_1\rangle^2\cdots \langle U_k,A_k\rangle^2dU_1\cdots dU_k=\|A_1\wedge\cdots\wedge A_k\|^2,
$$
where we write shortly $dU_i=\nu^{d-1}_{j_i}(dU_i)$ for the integration over $U_i\in G^{u_i^\perp}(d-1,j_i)$, and
on the right-hand side of the above equation the subspaces $A_i$ are considered as the associated unit simple multivectors.
\end{Lemma}

\begin{proof}
For given subspaces $U_i\in G^{u_i^\perp}(d-1,j_i)$, choose orthonormal bases\linebreak $\{ v^i_1,\ldots,v^i_{d-1}\}$ of $u_i^\perp$ so that
$$U_i=\Lin\{ v^i_1,\ldots,v^i_{j_i}\},\quad i=1,\ldots,k.$$
For numbers $0\leq p_i\leq j_i\wedge(d-1-j_i)$, define the function
\begin{equation}\label{eqn1}
\Phi_{u_1,\ldots,u_k}^{p_1,\ldots,p_k}(U_1,\ldots,U_k):=\sum_{I_1\in\cI^1_{p_1}}\dots\sum_{I_k\in\cI^k_{p_k}}
\left\|V^1_{I_1}\wedge\cdots \wedge V^k_{I_k}\right\|^2,
\end{equation}
where $V^i_I=\bigwedge_{l\in I}v^i_l$ and  $\cI^i_{p_i}$ denotes the family of all index sets $I\subset\{1,\ldots,d-1\}$ with $|I|=j_i$ and $|I\cap\{1,\ldots,j_i\}|=j_i-p_i$, $i=1,\ldots, k$. (The fact that $\Phi_{u_1,\ldots,u_k}^{p_1,\ldots,p_k}(U_1,\ldots,U_k)$ is a function of the subspaces $U_1,\ldots,U_k$ and does not depend on the choice of the orthonormal bases, follows similarly as in the proof of Proposition~2 in \cite{HRW} and is also implicitly contained in the argument below.)

We shall show that the function
\begin{equation}  \label{Phi}
\Phi_{u_1,\ldots,u_k}(U_1,\ldots,U_k):=\sum_{p_1}\dots\sum_{p_k}a_{p_1,\ldots,p_k}\Phi_{u_1,\ldots,u_k}^{p_1,\ldots,p_k}(U_1,\ldots,U_k)
\end{equation}
fulfills the requirement of the lemma for suitably chosen coefficients $a_{p_1,\ldots,p_k}$. The summation over $p_i$ runs from $0$ to $j_i\wedge(d-1-j_i)$, here and in the sequel ($i=1,\ldots,k$).

Define $\xi:=V^2_{I_2}\wedge\dots\wedge V^k_{I_k}$ and let $V\in G(d,d-j_1)$ be the linear subspace associated with $\xi$
if $\xi\neq 0$. If $\xi=0$, we choose $V\in G(d,d-j_1)$ arbitrarily. Then we have
\begin{eqnarray*}
\|V^1_{I_1}\wedge\dots\wedge V^k_{I_k}\|^2&=&\|V^1_{I_1}\wedge V\|^2\|\xi\|^2\\
&=&  \|V^1_{I_1}\wedge (V\cap u_1^\perp)\|^2\|p_{V^\perp}u_1\|^2\|\xi\|^2\\
&=&\langle V^1_{I_1},p_{u_1^\perp}(V^\perp)\rangle^2\|u_1\wedge \xi\|^2,
\end{eqnarray*}
which remains true also if $\xi=0$. Further, note that $p_{u_1^\perp}(V^\perp)$ equals the orthogonal complement of
$V\cap u_1^\perp$ in $u_1^\perp$ and that $\langle V^1_{I_1},p_{u_1^\perp}(V^\perp)\rangle^2=0$ if $V\subset u_1^\perp$.

We can thus write
\begin{eqnarray*}
\lefteqn{\Phi_{u_1,\ldots,u_k}^{p_1,\ldots,p_k}(U_1,\ldots,U_k)}\\&=&
\sum_{I_1\in\cI^1_{p_1}}\dots\sum_{I_k\in\cI^k_{p_k}} \langle V^1_{I_1},p_{u_1^\perp}(V^\perp)\rangle^2\|u_1\wedge \xi\|^2\\
&=&\sum_{I_2\in\cI^2_{p_2}}\dots\sum_{I_k\in\cI^k_{p_k}} \langle U_1,p_{u_1^\perp}(V^\perp)\rangle_{p_1}^2\|u_1\wedge \xi\|^2,
\end{eqnarray*}
where the $\ell$th product $\langle \cdot,\cdot\rangle_\ell$  on $G^{u_1^\perp}(d-1,j_1)$, for $0\le \ell\le j_1\wedge(d-1-j_1)$,
is defined and discussed in \cite[Section 5]{HRW}.
Integrating over $G^{u_1^\perp}(d-1,j_1)$, we obtain by \cite[Lemma~4]{HRW}
\begin{eqnarray*}
\lefteqn{\int \Phi_{u_1,\ldots,u_k}^{p_1,\ldots,p_k}(U_1,\ldots,U_k)\langle U_1,A_1\rangle^2\, dU_1}\\
&=&\sum_{q_1}d^{d-1,j_1}_{p_1,q_1}\sum_{I_2\in\cI^2_{p_2}}\dots\sum_{I_k\in\cI^k_{p_k}} \langle A_1,p_{u_1^\perp}V^\perp\rangle_{q_1}^2\|u_1\wedge \xi\|^2,
\end{eqnarray*}
for $q_1=0,\ldots,j_1\wedge(d-1-j_1)$,
with certain constants $d^{d-1,j_1}_{p_1,q_1}$.
If the coefficients $a_{p_1,\ldots,p_k}$ fulfill
$$\sum_{p_1}a_{p_1,\ldots,p_k}d^{d-1,j_1}_{p_1,q_1}=a_{p_2,\ldots,p_k}\text{ if }q_1=0\text{ and }0\text{ otherwise},$$
for some $a_{p_2,\ldots,p_k}$, we get
\begin{eqnarray*}
\lefteqn{\int \Phi_{u_1,\ldots,u_k}(U_1,\ldots,U_k)\langle U_1,A_1\rangle^2\, dU_1}\\
&=&\sum_{p_2}\dots\sum_{p_k}a_{p_2,\ldots,p_k}\sum_{I_2\in\cI^2_{p_2}}\dots\sum_{I_k\in\cI^k_{p_k}} \langle A_1,p_{u_1^\perp}(V^\perp)\rangle^2\|u_1\wedge \xi\|^2\\
&=&\sum_{p_2}\dots\sum_{p_k}a_{p_2,\ldots,p_k}\sum_{I_2\in\cI^2_{p_2}}\dots\sum_{I_k\in\cI^k_{p_k}} \| A_1\wedge V^2_{I_2}\wedge\dots\wedge V^k_{I_k}\|^2,
\end{eqnarray*}
which is also true if $V\subset u_1^\perp$ or if $\xi=0$, where  both sides of the equation are zero.

Continuing in the same way the integration with respect to $U_2,\ldots,U_k$, we get the desired solution, provided that
\begin{equation}  \label{E*}
\sum_{p_i}a_{p_i,\ldots,p_k}d^{d-1,j_i}_{p_i,q_i}=a_{p_{i+1},\ldots,p_k}\delta_{q_i,0},\quad q_i=0,\ldots, j_i\wedge(d-1-j_i),
\end{equation}
for $ i=1,\ldots, k-1$, with suitable coefficients $a_{p_i,\ldots,p_k}$ defined recursively. From \cite[Proposition~1]{HRW} we know that the matrices
$$D^{d-1}_{j_i}:=(d_{p,q}^{d-1,j_i})_{p,q=0}^{j_i\wedge(d-1-j_i)}$$
are regular for $i=1,\ldots,k$. Therefore, if we choose
$$a_{p_i,\ldots,p_k}=a^i_{p_i}\cdots a^k_{p_k}$$
with
$$(a^i_0,\ldots,a^i_{j_i\wedge(d-1-j_i)})=(1,0,\ldots,0)(D^{d-1}_{j_i})^{-1}$$
for $i=1,\ldots, k$,
then \eqref{E*} is satisfied and the proof is complete.
\end{proof}

\begin{proof}[Proof of Theorem \ref{thmapprox}]

Let  $\varepsilon>0$ be fixed.
Recall that
 $F^{(\varepsilon)}_n$ is nonnegative and  bounded from above on $(S^{d-1})^k$.

We define $c(d,n):= \gamma (d,n_1)\cdots \gamma (d,n_k)$ and
$$
\varphi_n(u_1,U_1,\ldots,u_k,U_k):=c(d,n)^{-1}\cdot\Phi_{u_1,\dots ,u_k }(U_1^\perp\cap u_1^\perp,\dots ,U_k^\perp\cap u_k^\perp).
$$
From \eqref{ECM},  we get
\begin{align*}
&\int_{F^\perp(d,d-1-n_1)}\cdots\int_{F^\perp(d,d-1-n_k)}F^{(\varepsilon)}_n(u_1,\dots, u_k)
\varphi_n(u_1,U_1,\ldots,u_k,U_k)\\
&\quad\times \Omega_{n_k}(K_k;d(u_k,U_k))\cdots \Omega_{n_1}(K_1;d(u_1,U_1))\\
&\ =  \int_{\nor(K_1)}\cdots\int_{\nor (K_k)}F^{(\varepsilon)}_n(u_1,\dots,u_k)\\
&\quad\times \sum_{|I_1|=d-1-n_1}\cdots \sum_{|I_k|=d-1-n_k} {\mathbb K}_{I_1}(K_1;x_1,u_1)\cdots {\mathbb K}_{I_k}(K_k;x_k,u_k)\\
&\quad\times \int_{G^{u_1^\perp}(d-1,d-1-n_1)}\cdots\int_{G^{u_k^\perp}(d-1,d-1-n_k)}
c(d,n)\varphi_n(u_1,U_1,\ldots,u_k,U_k)\,\\
&\quad\times\prod_{i=1}^k \langle U_i,A_{I_i}(K_i;x_i,u_i)\rangle^2  \, dU_k\cdots dU_1 \,
 {\cal H}^{d-1}(d(x_k,u_k))\cdots {\cal H}^{d-1}(d(x_1,u_1))\\
&\ =  \int_{\nor(K_1)}\cdots\int_{\nor (K_k)}F^{(\varepsilon)}_n(u_1,\dots,u_k)\\
&\quad\times \sum_{|I_1|=n_1}\cdots \sum_{|I_k|=n_k} {\mathbb K}_{I_1^c}(K_1;x_1,u_1)\cdots  {\mathbb K}_{I_k^c}(K_k;x_k,u_k)\\
&\quad\times \int_{G^{u_1^\perp}(d-1,d-1-n_1)}\cdots\int_{G^{u_k^\perp}(d-1,d-1-n_k)}
c(d,n)\varphi_n(u_1,U_1,\ldots,u_k,U_k)\\
&\quad\times\prod_{i=1}^k \langle U_i,A_{I_i^c}(K_i;x_i,u_i)\rangle^2 \, dU_k\cdots dU_1 \,{\cal H}^{d-1}(d(x_k,u_k))\cdots {\cal H}^{d-1}(d(x_1,u_1)).
\end{align*}
Since $\langle U_i,A_{I_i^c}(K_i;x_i,u_i)\rangle^2=\langle U_i^\perp\cap u_i^\perp,A_{I_i}(K_i;x_i,u_i)\rangle^2$
for $i=1,\ldots,k$, by the definition of $\varphi_n$  and Lemma \ref{Lem}, we deduce that
\begin{align*}
&\int_{F^\perp(d,d-1-n_1)}\cdots\int_{F^\perp(d,d-1-n_k)}F^{(\varepsilon)}_n(u_1,\dots, u_k)
\varphi_n(u_1,U_1,\ldots,u_k,U_k)\\
&\quad\times \Omega_{n_k}(K_k;d(u_k,U_k))\cdots \Omega_{n_1}(K_1;d(u_1,U_1))\\
&\ =  \int_{\nor(K_1)}\cdots\int_{\nor (K_k)}F^{(\varepsilon)}_n(u_1,\dots,u_k)\\
&\quad\times \sum_{|I_1|=n_1}\cdots \sum_{|I_k|=n_k} {\mathbb K}_{I_1^c}(K_1;x_1,u_1)\cdots  {\mathbb K}_{I_k^c}(K_k;x_k,u_k)\\
&\quad\times \int_{G^{u_1^\perp}(d-1,n_1)}\cdots\int_{G^{u_k^\perp}(d-1,n_k)}\Phi_{u_1,\ldots,u_k}(U_1,\dots ,U_k)\,\\
&\quad\times\prod_{i=1}^k \langle U_i,A_{I_i}(K_i;x_i,u_i)\rangle^2 \, dU_k\cdots dU_1 \,
 {\cal H}^{d-1}(d(x_k,u_k))\cdots {\cal H}^{d-1}(d(x_1,u_1))\\
&\ =  \int_{\nor(K_1)}\cdots\int_{\nor (K_k)}F^{(\varepsilon)}_n(u_1,\dots,u_k)\\
&\quad\times \sum_{|I_1|=n_1}\cdots \sum_{|I_k|=n_k} {\mathbb K}_{I_1^c}(K_1;x_1,u_1)\cdots  {\mathbb K}_{I_k^c}(K_k;x_k,u_k)\\
&\quad\times \| A_{I_1}(K_1;x_1,u_1)\wedge\dots\wedge A_{I_k}(K_k;x_k,u_k)\|^2 \,{\cal H}^{d-1}(d(x_k,u_k))\cdots {\cal H}^{d-1}(d(x_1,u_1))\\
&= \binom{d}{n_1\dots n_k} V^{(\varepsilon)}(K_1[n_1],\ldots,K_k[n_k]),
\end{align*}
which proves the asserted representation \eqref{flagrepeps}.
\end{proof}

From Theorem \ref{thmapprox} we now deduce the following limiting case under suitable assumptions of relative position.

\begin{Theorem}\label{mainthm}
Let $K_1, \dots ,K_k\subset\rd$  be convex bodies in $\rd$,  and let
$n=(n_1,\dots ,n_k)$ $\in\{1,\ldots,d-1\}^k$ with $n_1+\dots +n_k=d$.
Then, there is a continuous function $\varphi_{n}$ on $F^\perp(d,d-1-n_1)\times\cdots\times F^\perp(d,{d-1-n_k})$ (independent of $K_1,\dots ,K_k$) such that
\begin{align}\label{flagrep}
&V(K_1[n_1],\ldots,K_k[n_k])\ = \int_{F^\perp(d,{d-1-n_k})}\cdots\int_{F^\perp(d,d-1-n_1)} F_n(u_1,\dots ,u_k)\nonumber \\
&\quad \times\varphi_n(u_1,U_1,\dots, u_k,U_k)\, \Omega_{n_1}(K_1;d(u_1,U_1))\cdots \Omega_{n_k}(K_k;d(u_k,U_k))
\end{align}
holds
\begin{enumerate}
\item[{\rm (a)}] for $(\nu_d)^{k-1}$-almost all $(\rho_2,\dots,\rho_k)\in O(d)^{k-1}$, if $K_2,\dots ,K_k$ are replaced by $\rho_2 K_2,\dots ,\rho_kK_k$;
\item[{\rm (b)}] if all but one of the convex bodies $K_i$ have a support function of class $C^{1,1}$;
\item[{\rm (c)}] if $K_1,\dots ,K_k$ are
convex polytopes in general $(n_1,\dots ,n_k)$-position.
\end{enumerate}
\end{Theorem}

\begin{proof}  We choose $\varphi_n$ as in Theorem \ref{thmapprox}.
As pointed out before, Theorem \ref{mivol} and the monotone convergence theorem imply that
$$ V^{(\varepsilon)}(K_1[n_1],\ldots,K_k[n_k])\nearrow V(K_1[n_1],\ldots,K_k[n_k]),
\qquad \varepsilon\searrow 0.$$
Thus, in order to finish the proof of Theorem \ref{mainthm}, we have to show that
\begin{align*}
\lim_{\ve\to 0} &\int_{F^\perp(d,d-1-n_1)}\cdots\int_{F^\perp(d,d-1-n_k)}F^{(\varepsilon)}_n(u_1,\dots, u_k)
\varphi_n(u_1,U_1,\ldots,u_k,U_k)\\
&\quad\times \Omega_{n_k}(K_k;d(u_k,U_k))\cdots \Omega_{n_1}(K_1;d(u_1,U_1))\\
= &\int_{F^\perp(d,d-1-n_1)}\cdots\int_{F^\perp(d,d-1-n_k)}F_n(u_1,\dots, u_k)
\varphi_n(u_1,U_1,\ldots,u_k,U_k)\\
&\quad\times \Omega_{n_k}(K_k;d(u_k,U_k))\cdots \Omega_{n_1}(K_1;d(u_1,U_1))
\end{align*}
in each of the three cases listed in the theorem. For this, we use that $F_n^{(\ep)}\nearrow F_n$ as $\ep\searrow 0$ and verify that the dominated convergence theorem can be applied. The main step consists in finding a suitable upper bound for
$$
G := F_n(u_1,\dots,u_k)\cdot |\varphi_n(u_1,U_1,\ldots,u_k,U_k)| .
$$

\begin{Lemma}\label{upperbound} There is a constant $c\ge 0$ such that
$$
|\varphi_n(u_1,U_1,\ldots,u_k,U_k)|
\le c\, \|\underline{u} | L^\perp\|^2
$$
for all $(u_i,U_i)\in F^\perp(d,d-1-n_i)$, $i=1,\dots ,k$.
\end{Lemma}

\begin{proof} In view of the definition of the function $\varphi_n$, it is sufficient to show that
$$
\left\| V^1_{n_{1}}\wedge\cdots\wedge V^k_{n_k}\right\|\le d\left\|\underline{u}\vert L^\perp\right\|
$$
whenever $V^i_{n_i}=v^i_1\wedge\cdots\wedge v^i_{n_i}$, $i=1,\ldots,k$, $\{v_1^i,\ldots,v^i_{d-1},u_i\}$ is
an orthonormal basis of $\R^d$, and $n_1+\cdots+n_k=d$.

For this purpose, we put $\tau:=\max\{\|u_i-u_j\|:1\le i<j\le k\}$,
hence $\tau\le \sqrt{k}\,\|\underline{u}\vert L^\perp\|$.

If $\tau\ge 1$, then $\left\| V^1_{n_{1}}\wedge\cdots\wedge V^k_{n_k}\right\|\le 1\le \sqrt{k}\,\|\underline{u}\vert L^\perp\|$.

Now suppose that $\tau<1$. Let $v_d\in S^{d-1}$ be such that $\|v_d-u_i\| \le \tau $ for $i=1,\ldots,k$ (for instance $v_d:=u_1$), and let  $\{v_1,\ldots,v_d\}$ be an orthonormal basis of $\R^d$. Then there are unique $\alpha(i,j)\in\R$ and
$v^i_{j,\perp}\in v_d^\perp$ such that $v^i_j=v^i_{j,\perp}+\alpha(i,j)v_d$ for $i=1,\ldots,k$ and $j=1,\ldots,n_i$.
Thus, in particular, $\|v^i_{j,\perp}\|\le 1$ and $|\alpha(i,j)|=|\langle v^i_j,v_d\rangle |=|\langle v^i_j,v_d-u_i\rangle|\le \tau<1$. Hence, we obtain
$$
\left\| V^1_{n_{1}}\wedge\cdots\wedge V^k_{n_k}\right\|=\left\|
\bigwedge_{i=1}^k\bigwedge_{j=1}^{n_i}\left(v^i_{j,\perp}+\alpha(i,j)v_d\right)
\right\|\le d\tau,
$$
where we used that $n_1+\cdots+n_k=d$ and thus
$$
\bigwedge_{i=1}^k\bigwedge_{j=1}^{n_i} v^i_{j,\perp}
=0,
$$
which completes the proof.
\end{proof}

Hence $G$ can be bounded from above by
\begin{equation}\label{uppbound}
c\, \|\underline{u} | L^\perp\|^2 \int_{S^{k-1}_+}\|\ul{tu}|L^\perp\|^{-(k-1)d}\,\Ha^{k-1}(dt) .
\end{equation}
Define
$$S^{k-1}_\ast:=\left\{t\in S^{k-1}_+:\, t_i\geq\frac 1{2\sqrt{k}},\, i=1,\dots,k\right\}.$$

\begin{Lemma}\label{neulem}
Let $u_1,\dots,u_k\in S^{d-1}$.
\begin{enumerate}
\item[{\rm (1)}] If $t\in S^{k-1}_+\setminus S^{k-1}_\ast$, then $\|\ul{tu}|L^\perp\|\ge 1/(2k)$.
\item[{\rm (2)}] If $t\in  S^{k-1}_\ast$, then $\|\ul{tu}|L^\perp\|\ge \frac 1{2\sqrt{k}}\|\ul{u}|L^\perp\|$.
\end{enumerate}
\end{Lemma}

\begin{proof}
(1) If $t=(t_1,\dots,t_k)\in S^{k-1}_+\setminus S^{k-1}_\ast$, then  $t_j\geq \frac{1}{\sqrt{k}}$
for some $j\in\{1,\ldots,k\}$. In fact, otherwise we get $0<t_j< \frac 1{\sqrt{k}}$ for $j=1,\dots ,k$ and $0<t_i< \frac 1{2\sqrt{k}}$ for some $i\in\{1,\dots ,k\}$. Since $k\ge 2$, this would imply
$$ 1 = t_1^2+\dots +t_k^2\leq \frac{1}{4k}+(k-1)\frac 1k=\frac{4k-3}{4k}<1,$$
a contradiction. But then, for any  $t\in S^{k-1}_+\setminus S^{k-1}_\ast$ and $u_1,\dots,u_k\in S^{d-1}$, we have
$$\|\ul{tu}|L^\perp\|^2=\frac 1k\sum_{i<j}\|t_iu_i-t_ju_j\|^2\ge \frac 1k\left(\frac 1{\sqrt{k}} -\frac 1{2\sqrt{k}}\right)^2=\frac 1{4k^2},$$
which proves the first assertion.

(2) Now we assume that $t\in  S^{k-1}_\ast$. Let $i<j$. We distinguish two cases.

(a) If $\langle u_i,u_j\rangle \ge 0$, then
\begin{align*}
\|t_iu_i-t_ju_j\|^2 & = t_i^2+t_j^2-2t_it_j\langle u_i,u_j\rangle\\
&\ge t_i^2+t_j^2-(t_i^2+t_j^2)\langle u_i,u_j\rangle\\
&= (t_i^2+t_j^2) [1-\langle u_i,u_j\rangle]\\
&= \frac{1}{2}(t_i^2+t_j^2)\|u_i-u_j\|^2\\
&\ge \frac{1}{4k}\|u_i-u_j\|^2.
\end{align*}

(b) If $\langle u_i,u_j\rangle < 0$, then
\begin{align*}
\|t_iu_i-t_ju_j\|^2 & = t_i^2+t_j^2+2t_it_j(-\langle u_i,u_j\rangle)\\
&\ge \frac{1}{2k} +2\cdot\frac{1}{4k}(-\langle u_i,u_j\rangle)\\
&= \frac{1}{4k}\|u_i-u_j\|^2.
\end{align*}

Hence $\|t_iu_i-t_ju_j\|^2 \ge \frac{1}{4k}\|u_i-u_j\|^2$ for any $i<j$, from which the second assertion  follows.
\end{proof}

From \eqref{uppbound} and Lemma \ref{neulem}, we get
\begin{align*}
G&\le c\, 4k \int_{S^{k-1}_\ast} \|\ul{tu}|L^\perp\|^{2-(k-1)d}\,\Ha^{k-1}(dt)\\
& \qquad + c\, \|\ul{u}|L^\perp\|^{2} \int_{S^{k-1}_+\setminus S^{k-1}_\ast} (2k)^{(k-1)d}\,\Ha^{k-1}(dt),
\end{align*}
and the latter summand is bounded from above by a constant. Hence, we obtain
\begin{align*}
&\int_{F^\perp(d,d-1-n_1)}\cdots\int_{F^\perp(d,d-1-n_k)}|F_n(u_1,\dots, u_k) \varphi_n(u_1,U_1,\dots ,u_k,U_k)|\\
&\quad\times \Omega_{n_k}(K_k;d(u_k,U_k))\cdots \Omega_{n_1}(K_1;d(u_1,U_1))\\
&\le {\rm const.} \int_{S^{d-1}}\cdots\int_{S^{d-1}}\int_{S^{k-1}_\ast} \|\ul{tu}|L^\perp\|^{2-(k-1)d}\,\Ha^{k-1}(dt)\\
&\quad\times S_{n_k}(K_k,du_k)\cdots S_{n_1}(K_1,du_1)+ {\rm const.}\,,
\end{align*}
and we have to show, in each of the three cases (a), (b) und (c), that the latter integral is finite.

Let us first consider case (a). We apply independent uniform random orthogonal transformations $R_i\in O(d)$ to the bodies $K_i$, $i=2,\dots,k$, and observe that the mean area measure ${\mathbb E} S_{n_i}(R_iK_i,\cdot)$ is a finite rotation invariant measure on $S^{d-1}$. Using the upper bound for $G$, we see that it is sufficient to show that
\begin{equation}  \label{finite2}
\int_{S^{d-1}}\dots\int_{S^{d-1}}\int_{S^{k-1}_*} \|\ul{tu}|L^\perp\|^{2-(k-1)d}\,\Ha^{k-1}(dt)\, \Ha^{d-1}(du_k)\cdots \Ha^{d-1}(du_2)<\infty.
\end{equation}
Note that the last expression is  independent of $u_1\in S^{d-1}$. Hence, \eqref{finite2} is equivalent to
\begin{equation}  \label{finite3}
\int_{S^{d-1}}\dots\int_{S^{d-1}}\int_{S^{k-1}_\ast} \|\ul{tu}|L^\perp\|^{2-(k-1)d}\,\Ha^{k-1}(dt)\, \Ha^{d-1}(du_k)\cdots \Ha^{d-1}(du_1)<\infty.
\end{equation}
The mapping $g:(t,\ul{u})\mapsto\ul{tu}$ is one-to-one on $S^{k-1}_\ast\times(S^{d-1})^k$, its image is
$$S^{kd-1}_\Delta:=\left\{\ul{v}=(v_1,\dots,v_k)\in S^{kd-1}:\, \|v_i\|^2\geq\frac 1{4k},\, i=1,\dots,k\right\}$$
and the inverse map $h:=g^{-1}$ fulfills
\begin{align*}
\|h(\ul{v})-h(\ul{w})\|^2
&=\sum_{i=1}^k(\|v_i\|-\|w_i\|)^2+\sum_{i=1}^k\left\|\frac{v_i}{\|v_i\|}-\frac{w_i}{\|w_i\|}\right\|^2\\
&\leq \|\ul{v}-\ul{w}\|^2+\sum_{i=1}^k\frac 4{\|v_i\|^2}\|v_i-w_i\|^2\\
&\leq (1+16k)\|\ul{v}-\ul{w}\|^2,
\end{align*}
for $ \ul{v},\ul{w}\in S^{kd-1}_\Delta$.
Hence, $h$ is $\sqrt{1+16k}$-Lipschitz and its approximate Jacobian is bounded by ${\rm Lip}:=(1+16k)^{(kd-1)/2}$ from above. Consequently, the coarea formula yields
\begin{eqnarray*}
\lefteqn{\int_{S^{d-1}}\dots\int_{S^{d-1}}\int_{S^{k-1}_\ast} \|\ul{tu}|L^\perp\|^{2-(k-1)d}\,\Ha^{k-1}(dt)\, \Ha^{d-1}(du_k)\dots \Ha^{d-1}(du_1)}\hspace{2cm}\\
&=&\int_{S^{kd-1}_\Delta}\|\ul{v}|L^\perp\|^{2-(k-1)d}\, J_{kd-1}h(\ul{v})\, \Ha^{kd-1}(d\ul{v})\\
&\leq& {\rm Lip} \int_{S^{kd-1}}\|\ul{v}|L^\perp\|^{2-(k-1)d}\, \Ha^{kd-1}(d\ul{v}).
\end{eqnarray*}
The last integral is bounded by Lemma~\ref{fact}, hence \eqref{finite3} holds.

The case (b) is a consequence of \eqref{finite3}, since we may assume that $K_2,\dots ,K_k$ have  support functions of class $C^{1,1}$, and this implies that $S_{n_i}(K_i,\cdot )\le c_i\Ha^{d-1}$, with some constants $c_i, i=2,\dots ,k$ (this follows, for example, from \cite[Lemma 1]{HRW} together with \cite[Theorem 4.7]{Weil73}).

Finally, we treat case (c). Let $K_1,\dots ,K_k$ be convex polytopes in
general $(n_1,\dots ,n_k)$-position. Then we have
$$
\bigcap_{i=1}^k n(K_i,F_i) = \emptyset
$$
for all faces $F_i\in {\mathcal F}_{n_i}(K_i)$, where $n(K_i,F_i) = N(K_i,F_i) \cap S^{d-1}$, $i=1,\dots ,k$.
Consider the function
$$
f : S^{k-1}_{*} \times \bigtimes_{i=1}^kn(K_i,F_i) \to [0,\infty ) ,\quad (t,\ul{u}) \mapsto \|\ul{tu}|L^\perp\| .
$$
Clearly, $f$ is continuous and the domain of $f$ is compact. Moreover, $f>0$, since $f(t,\ul{u}) = 0$ implies that $t_iu_i=t_ju_j$ for all $i<j$, hence $t_i=t_j$ for all $i<j$. This yields $t_1=\dots =t_k=\frac 1{\sqrt{k}}$, and so $u_1 = \dots =u_k$ would be in $\bigcap_{i=1}^k n(K_i,F_i)$, a contradiction.

We obtain
$$
f(t,\ul{u})\ge \ve_0
$$
for some constant $\ve_0>0$ and all $(t,\ul{u})\in S^{k-1}_{*} \times \bigtimes_{i=1}^kn(K_i,F_i)$, and hence
\begin{align*}
\int_{S^{d-1}}&\cdots\int_{S^{d-1}}\int_{S^{k-1}_\ast} \|\ul{tu}|L^\perp\|^{2-(k-1)d}\,\Ha^{k-1}(dt)\\
&\quad\times S_{n_k}(K_k,du_k)\cdots S_{n_1}(K_1,du_1)\\
&={\rm const.} \sum_{F_1\in{\cal F}_{n_1}(K_1)}\cdots \sum_{F_k\in{\cal F}_{n_k}(K_k)}  \prod_{i=1}^k \Ha^{n_i}(F_i)\\
&\quad\times \int_{n(K_1,F_1)}\cdots\int_{n(K_k,F_k)}\int_{S^{k-1}_\ast} \|\ul{tu}|L^\perp\|^{2-(k-1)d}\,\Ha^{k-1}(dt)\\
&\quad\times\Ha^{d-1-n_k}(du_k)\cdots \Ha^{d-1-n_1}(du_1)\\
& <\infty ,
\end{align*}
since the integrand is bounded from above.

This concludes the proof of Theorem \ref{mainthm}, in each of the three cases.
\end{proof}

\section{Mixed translative functionals}\label{6}

We now consider, for $k\ge 2$, $j\in\{0,\dots,d-1\}$ and $r_1,\dots ,r_k\in \{ j,\dots, d\}$ with $r_1+\dots +r_k=(k-1)d+j$, a flag representation of the mixed functional $V_{r_1,\dots ,r_k}(K_1,\dots ,K_k)$. It is based on the following lemma, which  is the result corresponding to Lemma \ref{Lem}.

\begin{Lemma}\label{Lem2}
Let $u_1,\ldots,u_k\in S^{d-1}$ and $1\leq r_1,\ldots,r_k\leq d-1$ be given so that $r_1+\cdots+r_k{\geq}(k-1)d$. Then there exists a continuous function
$$\Psi_{u_1,\ldots,u_k}:\, G^{u_1^\perp}(d-1,d-1-r_1)\times\cdots\times G^{u_k^\perp}(d-1,d-1-r_k)\to\R$$
such that for all $A_1\in G^{u_1^\perp}(d-1,d-1-r_1),\ldots,A_k\in G^{u_k^\perp}(d-1,d-1-r_k)$,
\begin{align*}
\int\cdots\int\Psi_{u_1,\ldots,u_k}(U_1,\ldots,U_k)\, &\langle U_1,A_1\rangle^2 \cdots \langle U_k,A_k\rangle^2 \,
dU_1\cdots dU_k\\
&=\|A_1\wedge u_1\wedge\cdots\wedge A_k\wedge u_k\|^2,
\end{align*}
where  $dU_i=\nu^{d-1}_{d-1-r_i}(dU_i)$ denotes the  integration over $U_i\in G^{u_i^\perp}(d-1,d-1-r_i)$ with respect to the Haar probability measure, and the subspaces $A_i$  on the right-hand side of the above equation are considered as the associated unit simple multivectors.
\end{Lemma}

\begin{proof}
Put $j:=r_1+\dots +r_k-(k-1)d$. We shall first consider the case $j=0$.
The proof proceeds similarly as that of Lemma~\ref{Lem}.
For given subspaces $U_i\in G^{u_i^\perp}(d-1,d-1-r_i)$,
choose orthonormal bases $\{ v^i_1,\ldots,v^i_{d-1}\}$ of $u_i^\perp$ so that
$$U_i=\Lin\{ v^i_1,\ldots,v^i_{d-1-r_i}\},\quad i=1,\ldots,k.$$
For numbers $0\leq p_i\leq r_i\wedge(d-1-r_i)$, define the function
$$\Psi_{u_1,\ldots,u_k}^{p_1,\ldots,p_k}(U_1,\ldots,U_k)=\sum_{I_1\in\cI^1_{p_1}}\dots\sum_{I_k\in\cI^k_{p_k}}
\left\|V^1_{I_1}\wedge u_1\wedge\dots\wedge V^k_{I_k}\wedge u_k\right\|^2,$$
where $V^i_I=\bigwedge_{l\in I}v^i_l$ and $\cI^i_{p_i}$ denotes the family of all index sets $I\subset\{1,\ldots,d-1\}$ with $|I|=d-1-r_i$ and $|I\cap\{1,\ldots,d-1-r_i\}|=d-1-r_i-p_i$, $i=1,\ldots, k$. Again, $\Psi_{u_1,\ldots,u_k}^{p_1,\ldots,p_k}(U_1,\ldots,U_k)$ is a function of the subspaces $U_1,\ldots,U_k$ and does not depend on the choice of the orthonormal bases, as can be seen from the argument below.

We shall show that the function
\begin{equation}   \label{Psi}
\Psi_{u_1,\ldots,u_k}(U_1,\ldots,U_k)=\sum_{p_1}\dots\sum_{p_k}a_{p_1,\ldots,p_k}\Psi_{u_1,\ldots,u_k}^{p_1,\ldots,p_k}(U_1,\ldots,U_k)
\end{equation}
fulfills the requirement of the lemma for suitably chosen coefficients $a_{p_1,\ldots,p_k}$. Here and in the sequel, the summation over $p_i$ runs from $0$ to $r_i\wedge(d-1-r_i)$,  $i=1,\ldots,k$.

Denote $\xi:=V^2_{I_2}\wedge u_2\wedge\dots\wedge V^k_{I_k}\wedge u_k$ and let $V\in G(d,r_1)$ be the linear subspace associated with $\xi$ if $\xi\neq 0$. If $\xi=0$, we choose $V\in G(d,r_1)$ arbitrarily.  We have
\begin{equation}  \label{E2}
\|V^1_{I_1}\wedge u_1\wedge\dots\wedge V^k_{I_k}\wedge u_k\|^2=\langle V^1_{I_1},V^\perp\cap u_1^\perp\rangle^2\|u_1\wedge\xi\|^2 .
\end{equation}
Thus we can write
\begin{eqnarray*}
\Psi_{u_1,\ldots,u_k}^{p_1,\ldots,p_k}(U_1,\ldots,U_k)&=&
\sum_{I_1\in\cI^1_{p_1}}\dots\sum_{I_k\in\cI^k_{p_k}} \langle V^1_{I_1},V^\perp\cap u_1^\perp\rangle^2\|u_1\wedge \xi\|^2\\
&=&\sum_{I_2\in\cI^2_{p_2}}\dots\sum_{I_k\in\cI^k_{p_k}} \langle U_1,V^\perp\cap u_1^\perp\rangle_{p_1}^2\|u_1\wedge \xi\|^2,
\end{eqnarray*}
where again we refer to \cite[Section 5]{HRW} for a definition and the basic properties of the product $\langle\cdot,\cdot\rangle_{p_1}$.
Integrating over $G^{u_1^\perp}(d-1,d-1-r_1)$, we obtain by \cite[Lemma~4]{HRW}
\begin{eqnarray*}
\lefteqn{\int \Psi_{u_1,\ldots,u_k}^{p_1,\ldots,p_k}(U_1,\ldots,U_k)\langle U_1,A_1\rangle^2\, dU_1}\\
&=&\sum_{q_1= 0}^{r_1\wedge(d-1-r_1)}d^{d-1,d-1-r_1}_{p_1,q_1}\sum_{I_2\in\cI^2_{p_2}}\dots\sum_{I_k\in\cI^k_{p_k}} \langle A_1,V^\perp\cap u_1^\perp\rangle_{q_1}^2\|u_1\wedge \xi\|^2
\end{eqnarray*}
with constants $d^{d-1,d-1-r_1}_{p_1,q_1}$.
If the coefficients $a_{p_1,\ldots,p_k}$ fulfill
$$\sum_{p_1}a_{p_1,\ldots,p_k}d^{d-1,d-1-r_1}_{p_1,q_1}=a_{p_2,\ldots,p_k}\text{ if }q_1=0\text{ and }0\text{ otherwise},$$
we get
\begin{align}
&{\int \Psi_{u_1,\ldots,u_k}(U_1,\ldots,U_k)\langle U_1,A_1\rangle^2\, dU_1} \nonumber \\
&=\sum_{p_2}\dots\sum_{p_k}a_{p_2,\ldots,p_k}\sum_{I_2\in\cI^2_{p_2}}\dots\sum_{I_k\in\cI^k_{p_k}} \langle A_1,V^\perp\cap u_1^\perp\rangle^2\|u_1\wedge \xi\|^2\label{1int}\\
&=\sum_{p_2}\dots\sum_{p_k}a_{p_2,\ldots,p_k}\sum_{I_2\in\cI^2_{p_2}}\dots\sum_{I_k\in\cI^k_{p_k}} \| A_1\wedge u_1\wedge V^2_{I_2}\wedge u_2\wedge\dots\wedge V^k_{I_k}\wedge u_k\|^2 \nonumber
\end{align}
(we have used \eqref{E2} again in the last step), which remains true if $\xi=0$.
Continuing in the same way the integration with respect to $U_2,\ldots,U_k$, we get the desired solution, provided that
$$
\sum_{p_i}a_{p_i,\ldots,p_k}d^{d-1,d-1-r_i}_{p_i,q_i}=a_{p_{i+1},\ldots,p_k}\delta_{q_i,0},\quad q_i=0,\ldots,
r_i\wedge(d-1-r_i),
$$
for $ i=1,\ldots, k$.
The coefficients $a_{p_1,\ldots,p_k}$ can be found as in the proof of Lemma~\ref{Lem}.

It remains to treat the case $j>0$. Setting $r_{k+1}:=d-j$, we know by the first part of the proof
that for any $u_{k+1}\in S^{d-1}$ and any $A_i\in G^{u_i^\perp}(d-1,d-1-r_i)$, $i=1,\dots,k+1$, we have
\begin{align*}
 &\int\cdots\int\Psi_{u_1,\ldots,u_{k+1}}(U_1,\ldots,U_{k+1})\, \langle U_1,A_1\rangle^2 \cdots
\langle U_{k+1},A_{k+1}\rangle^2 \,
dU_1\dots dU_{k+1}\\
&\qquad \qquad=\|A_1\wedge u_1\wedge\cdots\wedge A_{k+1}\wedge u_{k+1}\|^2.
\end{align*}
If we integrate the expression on the right side with respect to the measure
$$\nu^{d-1}_{d-1-r_{k+1}}(dA_{k+1})\,\omega_d^{-1}\Ha^{d-1}(du_{k+1}),$$
which is a normalized invariant measure on $G(d,d-r_{k+1})$ and which thus agrees with $\nu^d_{d-r_{k+1}}$, we get
\begin{align}
&{\|A_1\wedge u_1\wedge\dots\wedge A_k\wedge u_k\|^2 \int_{G(d,d-r_{k+1})}\langle(A_1\wedge u_1\wedge\dots\wedge A_k\wedge u_k)^\perp, W\rangle^2\, dW}\nonumber\\
&\qquad =\binom{d}{j}^{-1}\|A_1\wedge u_1\wedge\dots\wedge A_k\wedge u_k\|^2 .\label{neu2eq}
\end{align}
Hence, the function
\begin{align*}
&{\Psi_{u_1,\dots,u_k}: (U_1,\dots,U_k)}\\
&\qquad\mapsto\binom{d}{j} \int_{S^{d-1}}\int_{G^{u_{k+1}^\perp}(d-1,d-1-r_{k+1})}\int_{G(d,d-r_{k+1})}
\Psi_{u_1,\ldots,u_{k+1}}(U_1,\ldots,U_{k+1})\\
&\qquad\qquad\qquad\times\langle U_{k+1},A_{k+1}\rangle^2\, dU_{k+1}\,dA_{k+1}\,\omega_d^{-1}\Ha^{d-1}(du_{k+1})
\end{align*}
fulfills the desired property. Moreover, we claim that it has again the form \eqref{Psi}.
Indeed, applying \eqref{1int} to $\Psi_{u_1,\dots,u_{k+1}}$ and the index $k+1$, we get
\begin{align*}
&{\int_{G(d,d-r_{k+1})} \Psi_{u_1,\ldots,u_{k+1}}(U_1,\ldots,U_{k+1})\,\langle U_{k+1},A_{k+1}\rangle^2\, dU_{k+1}}\\
&\qquad=\binom{d}{j}\sum_{p_1}\dots\sum_{p_k}a_{p_1,\ldots,p_k}\sum_{I_1\in\cI^1_{p_1}}\dots\sum_{I_k\in\cI^k_{p_k}}\\
&\qquad\qquad \times
\| V^1_{I_1}\wedge u_1\wedge\dots\wedge V^k_{I_k}\wedge u_k\wedge A_{k+1}\wedge u_{k+1} \|^2,
\end{align*}
and then performing the integration $dA_{k+1}\,\omega_d^{-1}\Ha^{d-1}(du_{k+1})$
and using the same argument as in \eqref{neu2eq}, we arrive at the form \eqref{Psi}.
\end{proof}

In order to prove a flag formula for mixed functionals, we first need a curvature representation, as in the case of mixed volumes. For the mixed (translative) functionals this has been obtained in \cite{HR}, in a local version and for sets of positive reach. Here, we only need the global version for convex bodies (we will come back to the local result in the next section). In the following, we put $r:= (r_1,\dots, r_k)$ and $j:=r_1+\cdots+r_k-(k-1)d\in\{0,\ldots,d-1\}$. Then this formula reads
\begin{eqnarray}\label{mixfunc}
&&V_{r}(K_1,\dots,K_k)\nonumber\\
&&\quad =\int_{\nor (K_1)\times\dots\times\nor (K_k)}G_{r}(u_1,\dots,u_k)\sum_{|I_1|=r_1,\dots,|I_k|=r_k}\left(\prod_{i=1}^k\bK_{I_i^c}(K_i;x_i,u_i)\right)\\
&&\quad\quad\times\left[ A_{I_1}(K_1;x_1,u_1) ,\ldots,A_{I_k}(K_k;x_k,u_k)
\right]^2\, \Ha^{k(d-1)}(d(x_1,u_1,\dots,x_k,u_k)),\nonumber
\end{eqnarray}
where
\begin{align}\label{G-function}
G_{r}(u_1,\dots,u_k) &:=\frac{1}{\omega_{d-j}}\int_{S^{k-1}_+}\left(\prod_{i=1}^kt_i^{d-1-r_i}\right)\left\|\sum_{i=1}^k t_iu_i\right\|^{-(d-j)}\Ha^{k-1}(dt)
\end{align}
for linearly independent $u_1,\dots ,u_k$ (and $G_r(u_1,\dots ,u_k)=0$ otherwise), and where
$$
\left[ A_{I_1}(K_1;x_1,u_1) ,\ldots,A_{I_k}(K_k;x_k,u_k)
\right]=\left\|\bigwedge_{i=1}^k A_{I_i^c}(K_i;x_i,u_i)\wedge u_1\wedge\cdots\wedge u_k
\right\|
$$
is the subspace determinant associated with the subspaces corresponding to $A_{I_i}(K_i;x_i,u_i)$, $i=1,\ldots,k$ (see
\cite[Section 2]{HR} for further references).

Note that the condition, which was imposed in \cite[Theorem 2]{HR} on the sets $K_1,\dots ,K_k$, is fulfilled for convex bodies, as was explained in \cite[Remark 1 (b)]{HR}.

As in the case of Theorem \ref{mainthm}, for $\ve>0$, we  introduce the bounded $\ve$-approximation
\begin{align*}
&V^{(\varepsilon)}_{r}(K_1,\ldots,K_k)\\
&\quad :=\int_{\nor (K_1)\times\dots\times\nor (K_k)}G_{r}^{(\ve)}(u_1,\dots,u_k)\sum_{|I_1|=r_1,\dots,|I_k|=r_k}\left(\prod_{i=1}^k\bK_{I_i^c}(K_i;x_i,u_i)\right)\\
&\quad\quad\times\left[ A_{I_1}(K_1;x_1,u_1) ,\ldots,A_{I_k}(K_k;x_k,u_k)
\right]^2\, \Ha^{k(d-1)}(d(x_1,u_1,\dots,x_k,u_k)),
\end{align*}
where now
$$
G^{(\varepsilon)}_r(u_1,\dots ,u_k) := G_r(u_1,\dots ,u_k) {\bf 1}\{ d(0,{\rm conv} \{u_1,\ldots,u_k\})\ge \ve\} .
$$
Then, $G^{(\varepsilon)}_r$ is nonnegative and  bounded from above on $(S^{d-1})^k$, since
$$
\|\sum_{i=1}^kt_iu_i\|\ge \|\sum_{i=1}^k(t_i/t_*)u_i\|\ge d(0,{\rm conv} \{u_1,\ldots,u_k\})\ge \ve
$$
with $t_*:=\sum_{i=1}^kt_i\ge \sum_{i=1}^kt_i^2=1$.

We put
$$
\psi_r(u_1,U_1,\ldots,u_k,U_k):=\tilde c(d,r)^{-1}\Psi_{u_1,\dots ,u_k }(U_1,\dots ,U_k)
$$
and $\tilde c(d,r):= \gamma (d,r_1)\cdots \gamma (d,r_k)$.
From \eqref{ECM} and Lemma \ref{Lem2}, we then get
\begin{align*}
&\int_{F^\perp(d,d-1-r_1)}\cdots\int_{F^\perp(d,d-1-r_k)}G^{(\varepsilon)}_r(u_1,\dots, u_k)
\psi_r(u_1,U_1,\ldots,u_k,U_k)\\
&\quad\times \Omega_{r_k}(K_k;d(u_k,U_k))\cdots \Omega_{r_1}(K_1;d(u_1,U_1))\\
&\ =  \int_{\nor(K_1)}\cdots\int_{\nor (K_k)}G^{(\varepsilon)}_r(u_1,\dots,u_k)\\
&\quad\times \sum_{|I_1|=d-1-r_1}\cdots \sum_{|I_k|=d-1-r_k} {\mathbb K}_{I_1}(K_1;x_1,u_1)\cdots
{\mathbb K}_{I_k}(K_k;x_k,u_k)\\
&\quad\times \int_{G^{u_1^\perp}(d-1,d-1-r_1)}\cdots\int_{G^{u_k^\perp}(d-1,d-1-r_k)}
\Psi_{u_1,\ldots,u_k}(U_1,\dots ,U_k)\,\\
&\quad\times\prod_{i=1}^k \langle U_i,A_{I_i}(K_i;x_i,u_i)\rangle^2\, dU_k\cdots dU_1 \,
 {\cal H}^{d-1}(d(x_k,u_k)\cdots {\cal H}^{d-1}(d(x_1,u_1))\\
&=  \int_{\nor(K_1)}\cdots\int_{\nor (K_k)}G^{(\varepsilon)}_r(u_1,\dots,u_k)\\
&\quad\times \sum_{|I_1|=d-1-r_1}\cdots \sum_{|I_k|=d-1-r_k}
{\mathbb K}_{I_1}(K_1;x_1,u_1)\cdots  {\mathbb K}_{I_k}(K_k;x_k,u_k)\\
&\quad\times \| A_{I_1}(K_1;x_1,u_1)\wedge u_1\wedge\dots\wedge A_{I_k}(K_k;x_k,u_k)\wedge u_k\|^2 \\
&\quad \times
{\cal H}^{d-1}(d(x_k,u_k))\cdots {\cal H}^{d-1}(d(x_1,u_1))\\
&=  \int_{\nor(K_1)}\cdots\int_{\nor (K_k)}G^{(\varepsilon)}_r(u_1,\dots,u_k) \sum_{|I_1|=r_1}\cdots \sum_{|I_k|=r_k}
\\
&\quad\times {\mathbb K}_{I_1^c}(K_1;x_1,u_1)\cdots
 {\mathbb K}_{I_k^c}(K_k;x_k,u_k)\left[ A_{I_1}(K_1;x_1,u_1) ,\ldots,A_{I_k}(K_k;x_k,u_k)
\right]^2 \\
&\quad\times {\cal H}^{d-1}(d(x_k,u_k))\cdots {\cal H}^{d-1}(d(x_1,u_1))\\
&=   V^{(\varepsilon)}_r(K_1,\ldots,K_k).
\end{align*}

The following theorem is the analog of Theorem \ref{mainthm} for mixed functionals. Also here, a condition of general position is needed. The cases (a) and (b) remain the same, but the notion of general position for polytopes has to be adapted. For $r_1,\dots, r_k$ with $r_1+\dots +r_k\ge (k-1)d$, we say that convex polytopes $K_1,\dots ,K_k\subset\rd$
are in general $(r_1,\dots ,r_k)$-position if
\begin{equation}\label{conda}
0\notin\text{conv}\{u_1,\ldots,u_k\}
\end{equation}
whenever
$u_i\in n(K_i,F_i)$ and  $F_i\in \mathcal{F}_{r_i}(K_i)$ for $i=1,\dots ,k$. Note that for $k=2$ and $r_1+r_2=d$ (where mixed functionals and mixed volumes are the same, up
to reflection of one of the bodies and a constant), the definition is consistent with the one used in Section \ref{5} (if we reflect one of the bodies).

\begin{Theorem}\label{mainthm2}
Let $K_1, \dots ,K_k\subset\rd$  be convex bodies in $\rd$, and let
$r=(r_1,\dots ,r_k) \in\{1,\ldots,d-1\}^k$ with $r_1+\dots +r_k\ge (k-1)d$. Then, there is a continuous function $\psi_{r}$ on $F^\perp(d,d-1-r_1)\times\cdots\times F^\perp(d,{d-1-r_k})$ (independent of $K_1,\dots ,K_k$) such that
\begin{align}\label{flagrep2}
&V_{r}(K_1,\ldots,K_k)\ = \int_{F^\perp(d,{d-1-r_k})}\cdots\int_{F^\perp(d,d-1-r_1)} G_r(u_1,\dots ,u_k)\nonumber \\
&\quad \times\psi_r(u_1,U_1,\dots, u_k,U_k)\, \Omega_{r_1}(K_1;d(u_1,U_1))\cdots \Omega_{r_k}(K_k;d(u_k,U_k))
\end{align}
holds
\begin{enumerate}
\item[{\rm (a)}] for $(\nu_d)^{k-1}$-almost all $(\rho_2,\dots, \rho_k)\in O(d)^{k-1}$, if $K_2,\dots ,K_k$ are replaced by $\rho_2K_2,\dots ,\rho_kK_k$;
\item[{\rm (b)}] if all but one of the convex bodies $K_i$ have a support function of class $C^{1,1}$.
\item[{\rm (c)}] if $K_1,\dots ,K_k$ are
convex polytopes in general $(r_1,\dots ,r_k)$-position;
\end{enumerate}
\end{Theorem}

\begin{proof}
For $\ve\searrow 0$, equation \eqref{mixfunc} implies that
$$
V^{(\varepsilon)}_{r}(K_1,\ldots,K_k)\nearrow  V_{r}(K_1,\ldots,K_k)
$$
for arbitrary convex bodies $K_1,\ldots,K_k\subset\R^d$.
As we have seen above,
\begin{align*}
&V^{(\varepsilon)}_{r}(K_1,\ldots,K_k) \\
&\quad =
\int_{F^\perp(d,d-1-r_1)}\cdots\int_{F^\perp(d,d-1-r_k)}G^{(\varepsilon)}_r(u_1,\dots, u_k)
\psi_r(u_1,U_1,\ldots,u_k,U_k)\\
&\quad\quad\times \Omega_{r_k}(K_k;d(u_k,U_k))\cdots \Omega_{r_1}(K_1;d(u_1,U_1)) .
\end{align*}
Thus,  we have to show that
\begin{align*}
\lim_{\ve\to 0} &\int_{F^\perp(d,d-1-r_1)}\cdots\int_{F^\perp(d,d-1-r_k)}G^{(\varepsilon)}_r(u_1,\dots, u_k) \
\psi_r(u_1,U_1,\ldots,u_k,U_k)\\
&\quad\times \Omega_{n_k}(K_k;d(u_k,U_k))\cdots \Omega_{n_1}(K_1;d(u_1,U_1))\\
= &\int_{F^\perp(d,d-1-r_1)}\cdots\int_{F^\perp(d,d-1-r_k)}G_r(u_1,\dots, u_k) \psi_r(u_1,U_1,\ldots,u_k,U_k)\\
&\quad\times \Omega_{n_k}(K_k;d(u_k,U_k))\cdots \Omega_{n_1}(K_1;d(u_1,U_1)),
\end{align*}
in each of the three cases listed in the theorem. As in the case of Theorem \ref{mainthm}, we have to discuss the integrability of suitable upper bounds for
$$G_r(u_1,\dots, u_k) |\psi_r(u_1,U_1,\ldots,u_k,U_k)|.$$

Recall that $G_r^{(\ep)}\nearrow G_r$ as $\ep\searrow 0$. For $\psi_r$ we use the following lemma.

\begin{Lemma}\label{upperbound2} There is a constant $c\ge 0$ such that
$$
|\psi_r(u_1,U_1,\dots ,u_k,U_k)|\le c\, \|u_1\wedge\cdots\wedge u_k\|^2
$$
for all $(u_i,U_i)\in F^\perp(d,d-1-n_i), i=1,\dots ,k$.
\end{Lemma}

This follows from the definition of $\Psi_{u_1,\dots ,u_k }^{p_1,\dots ,p_k}$ as a finite sum of expressions of the
form $\| V_{I_1}^1\wedge u_1\wedge \cdots \wedge V_{I_k}^k\wedge u_k\|^2$, each of which is
bounded from above by $\|u_1\wedge\cdots\wedge u_k\|^2$.

Recall that $j=r_1+\cdots+r_k-(k-1)d\le k(d-1)-(k-1)d=d-k$. Concerning the upper estimate for
$$
J:=\|  u_1\wedge \cdots \wedge u_k\|^2\int_{S^{k-1}_+} \|\sum_{i=1}^k t_iu_i\|^{-(d-j)} \,\Ha^{k-1}(dt)
$$
in the cases (a) and (b), we first use the upper bound from \cite[Lemma 3, (13)]{HR} to see that $J\le \text{const}$ if
$j=d-k$ and $J\le \text{const} \|u_1\wedge\cdots\wedge u_k\|^{-(d-k)}$ if $j<d-k$. In the latter case, we can then
argue as in the proof of \cite[Proposition 1]{HR}.

For case (c), assume that $K_1,\ldots,K_k$ are in general $(r_1,\dots, r_k)$-position. By a compactness and continuity
argument  this means  that there is a positive constant $\varepsilon_0>0$ such that
$$
 \left\|\sum_{i=1}^k s_iu_i\right\|\ge\ve_0>0
$$
holds for all $s=(s_1,\dots, s_k)\in [0,1]^k$ with $s_1+\cdots+s_k=1$,
$u_i\in n(K_i,F_i)$ and  $F_i\in \mathcal{F}_{r_i}(K_i)$, for $i=1,\dots ,k$. This again holds if and only
if there is a positive constant $\varepsilon_1>0$ such that
$$\left\|\sum_{i=1}^k t_iu_i\right\|\ge\ve_1>0$$
holds for all $t=(t_1,\dots, t_k)\in S^{k-1}_+$,
$u_i\in n(K_i,F_i)$ and  $F_i\in \mathcal{F}_{r_i}(K_i)$, for $i=1,\dots ,k$. The latter clearly
 guarantees the integrability.
\end{proof}

\noindent{\bf Remark.} If $K_1,\ldots,K_k$ are polytopes with nonempty interiors, then
\begin{equation}\label{condb}
0\notin\text{conv}\left(n(K_i,F_i)\right)
\end{equation}
whenever $F_i\in \mathcal{F}_{r_i}(K_i)$ for $i\in\{1,\dots ,k\}$.
Assuming \eqref{condb}, if follows that \eqref{conda} is equivalent to requiring that
\begin{equation}\label{condc}
0\notin\text{conv}\left(\bigcup_{i=1}^k n(K_i,F_i)\right)
\end{equation}
whenever $F_i\in \mathcal{F}_{r_i}(K_i)$ for $i=1,\dots ,k$.

\section{Mixed curvature measures}\label{7}

To derive a flag representation for the mixed curvature measures of translative integral geometry, our starting point is a curvature representation of the mixed curvature measures (see \cite[Theorem 2]{HR}), which states that
\begin{eqnarray*}
&&C_{r}(K_1,\dots,K_k; A)\nonumber\\
&&\quad =\int_{\nor (K_1)\times\dots\times\nor (K_k)}G_{r}(\underline{(x,u)};A)\sum_{|I_1|=r_1,\dots,|I_k|=r_k}
\left(\prod_{i=1}^k\bK_{I_i^c}(K_i;x_i,u_i)\right)\\
&&\quad\quad\times\left[ A_{I_1}(K_1;x_1,u_1) ,\ldots,A_{I_k}(K_k;x_k,u_k)
\right]^2\, \Ha^{k(d-1)}(d(x_1,u_1,\dots,x_k,u_k)),\nonumber
\end{eqnarray*}
where $r=(r_1,\dots ,r_k)$, $ A \subset\R^{kd}\times S^{d-1}$ is a Borel set,
\begin{align*}
G_{r}(\underline{(x,u)};A)  &:=\frac{1}{\omega_{d-j}}\int_{S^{k-1}_+}
\mathbf{1}_A(\underline{x},\underline{u}(t))
\left(\prod_{i=1}^kt_i^{d-1-r_i}\right)\left\|\tilde{u}(t)\right\|^{-(d-j)}\Ha^{k-1}(dt)
\end{align*}
for linearly independent $u_1,\dots ,u_k$ (and $G_{r}(\underline{(x,u)};A) =0$ otherwise), for $j:=r_1+\cdots +r_k-(k-1)d$, and where
$$
\tilde{u}(t):=\sum_{i=1}^kt_iu_i\qquad\text{and}\qquad \underline{u}(t):=\frac{\tilde{u}(t)}{\|\tilde{u}(t)\|}.
$$
As before, for $\ve>0$ we  introduce the bounded $\ve$-approximation
\begin{align*}
&C_{r}^{(\varepsilon)}(K_1,\dots,K_k; A)\\
&\quad :=\int_{\nor (K_1)\times\dots\times\nor (K_k)}G_{r}^{(\ve)}(\underline{(x,u)};A)\sum_{|I_1|=r_1,\dots,|I_k|=r_k}\left(\prod_{i=1}^k\bK_{I_i^c}(K_i;x_i,u_i)\right)\\
&\quad\quad\times\left[ A_{I_1}(K_1;x_1,u_1) ,\ldots,A_{I_k}(K_k;x_k,u_k)
\right]^2\, \Ha^{k(d-1)}(d(x_1,u_1,\dots,x_k,u_k)),
\end{align*}
where now
$$
G_{r}^{(\ve)}(\underline{(x,u)};A) := G_{r}(\underline{(x,u)};A) {\bf 1}\{ d(0,{\rm conv} \{u_1,\ldots,u_k\})\ge \ve\} .
$$
Clearly,  $G^{(\varepsilon)}_r$ is nonnegative and  bounded from above by $\omega_{d-j}^{-1}\ve^{-(d-j)}\omega_k$,
independent of $(\underline{(x,u)};A)$.

For the flag representations of mixed curvature measures, we need an extension of the
flag measures $\Omega_k(K;\cdot)$  which we briefly recall.
 In the following, we consider the {\it flag manifold} $\Forth_*(d,k):= \R^d \times \Forth(d,k)$.
For a convex body $K\subset\rd$ and $k\in\{0,\dots ,d-1\}$, the $k$th {\it flag measure} $\Gamma_k(K;\cdot)$ of $K$  is a measure on $\Forth_*(d,d-1-k)$ defined by
\begin{align*}
&\int g(x,u,V)\, \Gamma_k(K;d(x,u,V))\\
&\qquad =\gamma(d,k)\int_{\nor (K)}\sum_{|I|=
\bark}\BK_I(K;x,u)\int_{G^{u^\perp}(d-1,d-1-k)} g(x,u,V)
\\
&\qquad \qquad\times\, \langle V,A_I(K;x,u)\rangle^2\,
\nu^{d-1}_{d-1-k}(dV)\,{\cal H}^{d-1}(d(x,u)).
\end{align*}

As in the preceding section, we put
$$
\psi_r(u_1,U_1,\ldots,u_k,U_k):=\tilde c(d,r)^{-1}\Psi_{u_1,\dots ,u_k }(U_1,\dots ,U_k)
$$
and $\tilde c(d,r):= \gamma (d,r_1)\cdots \gamma (d,r_k)$.

Repeating the reasoning of the preceding section, we obtain
\begin{align*}
&C_{r}^{(\varepsilon)}(K_1,\dots,K_k; A)\\
&\qquad =\int_{F_*^\perp(d,d-1-r_k)}\cdots\int_{F_*^\perp(d,d-1-r_1)}G_{r}^{(\ve)}(\underline{(x,u)};A)
\psi_r(u_1,U_1,\ldots,u_k,U_k)\\
&\qquad\qquad \times \Gamma_{r_1}(K_1;d(x_1,u_1,U_1))\cdots \Gamma_{r_k}(K_k;d(x_k,u_k,U_k)) ,
\end{align*}
where  $ A \subset\R^{kd}\times S^{d-1}$ is a Borel set.

We also obtain as an immediate consequence the following result.

\begin{Theorem}\label{mainthm2b}
Let $K_1, \dots ,K_k\subset\rd$  be convex bodies in $\rd$,  let
$r=(r_1,\dots ,r_k) \in\{1,\ldots,d-1\}^k$ with $r_1+\dots +r_k\ge (k-1)d$,
and let  $ A \subset\R^{kd}\times S^{d-1}$ be a Borel set. Then, there is a continuous function $\psi_{r}$ on $F^\perp(d,d-1-r_1)\times\cdots\times F^\perp(d,{d-1-r_k})$ (independent of $K_1,\dots ,K_k$) such that
\begin{align*}
&C_{r}(K_1,\dots,K_k; A)\ = \int_{F_*^\perp(d,{d-1-r_k})}\cdots\int_{F_*^\perp(d,d-1-r_1)} G_{r} (\underline{(x,u)};A)\nonumber \\
&\quad \times \psi_r(u_1,U_1,\dots, u_k,U_k)\, \Gamma_{r_1}(K_1;d(x_1,u_1,U_1))\cdots \Gamma_{r_k}(K_k;d(x_k,u_k,U_k))
\end{align*}
holds under any of the conditions (a) -- (c) in Theorem \ref{mainthm2}.
\end{Theorem}

\end{document}